\documentclass[11pt,letterpaper]{amsart}
\usepackage[utf8]{inputenc}
\usepackage{amsmath,amsfonts,amsthm,amssymb,verbatim,etoolbox,color,bm}
\usepackage{stmaryrd}

\newcommand{\bbz}{\mathbb{Z}}

\newcommand{\abs}[1]{\left\lvert #1\right\rvert}

\newcommand{\brac}[1]{\left( #1\right)}

\newcommand{\norm}[1]{\left\lVert #1\right\rVert}

\newcommand{\Mod}[1]{\ (\mathrm{mod}\ #1)}
\newcommand{\subsum}[1]{\sum_{\substack{#1}}}

\newcommand{\ogcd}[2]{\gcd(#1,#2)}
\newcommand{\bracgcd}[2]{\gcd\brac{#1,#2}}
\newtheorem{theorem}{Theorem}
\newtheorem{lemma}{Lemma}

\newtheorem{conjecture}{Conjecture}

\theoremstyle{definition}
\newtheorem{definition}{Definition}

\title{Odd moments and adding fractions}
\author{Thomas F. Bloom and Vivian Kuperberg}
\begin{document}

 \begin{abstract}
We prove near-optimal upper bounds for the odd moments of the distribution of coprime residues in short intervals, confirming a conjecture of Montgomery and Vaughan. As an application, we prove near-optimal upper bounds for the average of the refined singular series in the Hardy-Littlewood conjectures concerning the number of prime $k$-tuples for $k$ odd. The main new ingredient is a near-optimal upper bound for the number of solutions to $\sum_{1\leq i\leq k}\tfrac{a_i}{q_i}\in \mathbb{Z}$ when $k$ is odd, with $\ogcd{a_i}{q_i}=1$ and restrictions on the size of the numerators and denominators, which is of independent interest.
 \end{abstract}
\maketitle

\section{Introduction}
In this paper, we prove two new results related to odd moments in number theory, which have received much less attention than their even counterparts. Both of our upper bounds are within logarithmic factors of the conjectured truth. Our main new tool is a new, essentially best possible, upper bound for the number of solutions to a rational linear equation with an odd number of summands.

\subsection{Distribution of reduced residues}
Let $q,h\geq 1$ and, for any integer $k\geq 1$, let
\[M_k(q,h) = \sum_{1\leq n\leq q}\brac{\sum_{\substack{n\leq m < n+h\\ \ogcd{m}{q}=1}}1-\frac{\phi(q)}{q}h}^k.\]
That is, if we choose an interval of length $h$ uniformly at random, then $M_k(q,h)$ is the $k$th central moment of the proportion of the interval which is coprime to $q$ (which has expected value $\tfrac{\phi(q)}{q}h$).

Montgomery and Vaughan \cite{MV86} showed that for any fixed integer $k\geq 1$\footnote{In this paper we use the Vinogradov $f\ll g$ notation to mean $f=O(g)$, that is, there exists some absolute constant $C$ such that $\lvert f\rvert \leq Cg$. All implied constants depend (only, yet substantially) on $k$.}
\[M_k(q,h)\ll \brac{\frac{\phi(q)}{q}h+\brac{\frac{\phi(q)}{q}h}^{\frac{k}{2}}}q.\]
This is sharp for even $k$, but not for odd $k$. For odd $k$ Montgomery and Vaughan proved the stronger bound (see \cite[equation (18)]{MV86})\footnote{In fact the second summand here always dominates the first, but we write it in this form for easier comparison to the other bounds discussed.}
\[M_k(q,h)\ll \brac{\frac{\phi(q)}{q}h+\brac{\frac{\phi(q)}{q}h}^{\frac{k}{2}-\frac{1}{7k}}\brac{\frac{q}{\phi(q)}}^{2^k-\frac{1}{7k}}}q ,\]
but speculated that a much stronger bound should be true, with the exponent in the second summand replaced by $\frac{k-1}{2}$ (when $k$ is odd). In this paper we (almost) prove this conjecture, losing only a power of $\log h$.
\begin{theorem}\label{thm:montgomery-vaughan-odd}
Let $q \ge 1$ and $h\ge 2$. For any odd $k\geq 3$
\[M_k(q,h)\ll (\log h)^{O(1)}\brac{\frac{\phi(q)}{q}h+\brac{\frac{\phi(q)}{q}h}^{\frac{k-1}{2}}}q.\]
\end{theorem}

In fact, the method of Montgomery and Vaughan allows for an asymptotic formula for $M_k$ when $k$ is even (this was sketched in \cite{MV86} and proved in detail by Montgomery and Soundararajan \cite{MS04}). We are unable to prove an asymptotic formula for $M_k$ when $k$ is odd, although we believe that one should be possible. It is not clear even what the leading constant of the main term of such an asymptotic should be (unlike in the even case when there is a clearly dominant `diagonal term'). Nonetheless, we believe (and computational experimentation suggests) that there is a main term of order $(\frac{\phi(q)}{q}h+(\tfrac{\phi(q)}{q}h)^{\frac{k-1}{2}})q$, and hence we believe that Theorem~\ref{thm:montgomery-vaughan-odd} is essentially the best possible, in that it is only improvable by a factor of $(\log h)^{O(1)}$. Recent work of La Bret\`eche \cite{dlB-3rd-moment} provides an asymptotic estimation of $R_3(h)$ (see Section \ref{subsec:introduction:distribution-of-prime-tuples}) which is consistent with this analysis and suggests that an asymptotic may be within reach, whereas in recent work Gorodetsky \cite[Conjecture 1]{Gor} argues that odd moments of reduced residues (or primes) in short intervals can be modeled by odd moments of a centered Poisson random variable, which predicts an asymptotic for $M_k(q,h)$ of the shape $(\tfrac{\phi(q)}{q}h+c_k(\tfrac{\phi(q)}{q}h)^{\frac{k-1}{2}})q$, where $c_k$ is an explicit constant related to the Poisson distribution. 

\subsection{Distribution of prime $k$-tuples}\label{subsec:introduction:distribution-of-prime-tuples}
For any $k$-tuple $d_1,\ldots,d_k$ of distinct integers one can ask about the count of $n\leq N$ such that $n+d_i$ are prime for all $1\leq i\leq k$. Hardy and Littlewood \cite{HL} gave a convincing heuristic that produces a precise conjecture for the asymptotics of such a count. As usual when dealing with primes, the statements are simplified if we consider a weighted count instead, replacing $1_{\mathrm{prime}}$ by the smoother von Mangoldt function $\Lambda$, defined by $\Lambda(n)=\log p$ if $n$ is a power of a prime $p$ and $0$ otherwise. In this form, the Hardy-Littlewood conjecture states that
\[\frac 1N \sum_{n\leq N}\prod_{i=1}^k \Lambda(n+d_i)= \mathfrak{S}(d_1,\ldots,d_k)+o(1)\]
(for any fixed $k$-tuple $(d_1,\ldots,d_k)$ of distinct integers, as $N\to \infty$), where
\[ \mathfrak{S}(d_1,\ldots,d_k) = \prod_{p}\brac{1-\frac{1}{p}}^{-k}\brac{1-\frac{\nu_p(d_1,\ldots,d_k)}{p}}\]
and for any prime $p$ we write $\nu_p(d_1,\ldots,d_k)$ for the number of residue classes modulo $p$ occupied by the $d_1,\ldots,d_k$.  Gallagher \cite{Ga} proved the asymptotic that for any fixed $k\geq 1$,
\[\sum_{\substack{1\leq d_1,\ldots,d_k\leq h\\ d_i\textrm{ distinct}}}\mathfrak{S}(d_1,\ldots,d_k)\sim h^k.\]
(That is, on average $\mathfrak{S}$ tends to $1$ as $h\to \infty$.) 

This so-called `singular series' $\mathfrak{S}$ is an object of great interest in the study of primes -- aside from the obvious interest in counting prime tuples, the study of $\mathfrak{S}$ is also important in understanding the distribution of primes in short intervals (see for example \cite{MS04}) and in predicting the congruence classes of consecutive primes (see for example \cite{LOS16}).

Since, by the prime number theorem, $\frac{1}{N}\sum_{n\leq N}\Lambda(n)\sim 1$, for more precise information one should subtract the `main term' $1$ from each $\Lambda$. The corresponding refined form of the Hardy-Littlewood conjecture then states that
\[\frac 1N \sum_{n\leq N}\prod_{i=1}^k (\Lambda(n+d_i)-1)=\mathfrak{S}_0(d_1,\ldots,d_k) + o(1)\]
where for $\mathcal D = (d_1, \ldots, d_k),$
\begin{equation*}
\mathfrak{S}_0(\mathcal D) = \sum_{X\subseteq \{1,\ldots,k\}}(-1)^{k-\abs{X}}\mathfrak{S}(\mathcal D_X).
\end{equation*}
Here we use the notational convention that $\mathcal D_X$ denotes the tuple consisting only of those elements of $\mathcal D$ with indices in $X$.
Using $\mathfrak{S}(\mathcal D)=\sum_{X\subseteq \mathcal D} \mathfrak{S}_0(\mathcal D_X)$ one can recover information about $\mathfrak{S}$ from asymptotics for $\mathfrak{S}_0$, but the latter yields more information, and the refined singular series $\mathfrak{S}_0$ is more difficult to study. (In particular note that, unlike $\mathfrak{S}$, the refined series $\mathfrak{S}_0$ can take negative values.)

Montgomery and Soundararajan \cite{MS04} considered the sum
\[R_k(h) = \sum_{\substack{1\leq d_1,\ldots,d_k\leq h\\ d_i\textrm{ distinct}}}\mathfrak{S}_0(d_1,\ldots,d_k).\]
When $k$ is even 
Montgomery and Soundararajan \cite{MS04} proved an asymptotic for $R_k(h)$ of the shape
\[R_k(h)\sim \mu_k (-h \log h)^{k/2}\]
for some explicit constant $\mu_k>0$, which implies Gallagher's estimate for the average of $\mathfrak{S}$. When $k$ is odd they proved, for any $\epsilon>0$,
\[R_k(h) \ll_{\epsilon} h^{\frac{k}{2}-\frac{1}{7k}+\epsilon}.\]
It has been conjectured, by Lemke Oliver and Soundararajan in \cite{LOS16} and by the second author in \cite{K21}, that when $k\geq 3$ is an odd integer and $h>1$ we have
\[R_k(h) \asymp (\log h)^{\frac{k+1}{2}}h^{\frac{k-1}{2}}.\]
In \cite{K21} the second author obtained bounds of comparable quality (except for the power of $\log h$) when $k=5$ in the analogous function field setting. We prove bounds of the conjectured quality for the integers, for all odd $k\geq 3$.
\begin{theorem}\label{thm:conjecture-for-Rkh-for-k-odd}
If $k\geq 3$ is an odd integer and $h \ge 2$ then
\[R_k(h)\ll (\log h)^{O(1)} h^{\frac{k-1}{2}}.\]
\end{theorem}

As with Theorem~\ref{thm:montgomery-vaughan-odd}, and in keeping with the previous conjecture for the growth of $R_k(h)$, we believe (but cannot prove) that this upper bound is essentially the best possible, up to the exponent of $\log h$.

Recently, La Bret\`eche \cite{dlB-3rd-moment} has given an asymptotic for $k=3$, showing that $R_3(h) \sim \frac 92 h (\log h)^2$. Moreover, La Bret\`{e}che gives a precise conjecture that, for all integers $\ell\geq 1$,
\[R_{2\ell+1}(h)\sim (-1)^{\ell-1}(2\ell+1)\frac{3(2\ell)!}{2^{\ell+1}(\ell-1)!}(\log h)^{\ell+1}h^{\ell}.\]

In Section~\ref{subsec-rkformulae} we calculate exact expressions for the expected main terms of $R_k(h)$ in terms of $M_j(q,h)$ (via an auxiliary expression $V_j(q,h)$). As mentioned above we have no good heuristics for the asymptotics for $M_j(q,h)$, so these expressions are currently of limited value. The asymptotics for $M_j(q,h)$ are, heuristically, closely related to the asymptotics for the number of primes in short intervals. The three problems of estimating odd moments for primes in short intervals, $M_k(q,h)$ for odd $k$, and $R_k(h)$ for odd $k$, are strongly linked; see \cite{dlB-3rd-moment} for further information.

\subsection{Adding fractions}
The main new idea is an improved upper bound for the number of solutions to a rational linear equation with an odd number of summands, of independent interest, which we expect will have other applications. The bound we prove is fairly technical, but the following simpler consequence is representative.

\begin{theorem}\label{th-simpfrac}
Let $Q \ge 2$ and $1\leq n\leq Q$ be integers, and let $k\geq 3$ be an odd integer. The number of solutions to 
\[\frac{a_1}{q_1}+\cdots+\frac{a_{k}}{q_{k}}\in \bbz\quad\textrm{ where }\quad\ogcd{a_i}{q_i}=1\textrm{ for }1\leq i\leq k,\] 
with $\abs{a_i}\leq n$ and $q_i\leq Q$ for $1\leq i\leq k$, is
\[\ll (\log Q)^{O(1)} n^{\frac{k+1}{2}}Q^{\frac{k-1}{2}}.\]
\end{theorem}

A bound of the shape of Theorem~\ref{th-simpfrac} was conjectured by the second author in \cite[Conjecture 6.2]{K21}. Similar bounds have been previously established for some specific ranges of $n$, $Q$, and $k$:
\begin{itemize}
    \item Blomer and Br\"{u}dern \cite[Lemma 3]{BB} prove Theorem~\ref{th-simpfrac} when $k=3$, establishing an upper bound of $\ll (\log Q)^{O(1)}n^2Q$. (The result of \cite{BB} has slightly different coprimality conditions, but their proof easily adapts to the situation in Theorem~\ref{th-simpfrac}.)
    \item Blomer, Br\"{u}dern, and Salberger \cite{BBS} prove an asymptotic formula for the number of such solutions when $n=Q$ and $k=3$, under the weaker coprimality assumption\footnote{In other words, they study solutions to this equation in $\mathbb{P}^5$ rather than $\mathbb{P}^1\times \mathbb{P}^1\times \mathbb{P}^1$, as our coprimality conditions assume. This difference should only affect the power of $\log Q$ in the count of solutions, and hence is essentially immaterial at our present level of precision.} that there is no common divisor of $(a_1,a_2,a_3,q_1,q_2,q_3)$. The same authors \cite{BBS18} also prove an asymptotic formula when $n=Q$ and $k=3$ under the coprimality assumption that there is no common divisor of $(a_1,a_2,a_3)$ and no common divisor of $(q_1,q_2,q_3)$. 
    \item An asymptotic formula when $n = Q$ and $k = 3$ under the same coprimality assumption that we consider, namely that $(a_i,q_i) = 1$ for all $i$, was proven by Bettin and Destagnol \cite{BD}.
    \item An asymptotic formula for general $k$, again in the special case when $n=Q$ and under the weaker coprimality condition that $(a_1,\ldots,q_k)$ have no common divisor, was established by La Bret\`eche \cite{dlB} and Destagnol \cite{De}.
    \item Shparlinski \cite{Shpar} proves a version of Theorem~\ref{th-simpfrac}, for arbitrary $k$, under the condition that $n = Q$; for his result, the constraint that $\sum_i a_i/q_i \in \mathbb Z$ is weighted, so that for fixed $w_1, \ldots, w_k$, one has $\sum_i w_i a_i/q_i \in \mathbb Z$. This result was improved by La Bret\`eche \cite{dlB} who improved the power of log.
    \item When $k$ is even, the analogue of Theorem~\ref{th-simpfrac} is well-known, with the upper bound for the number of solutions being $\ll (\log Q)^{O(1)}n^{\frac{k}{2}}Q^{\frac{k}{2}}$, which the trivial diagonal contribution shows is best possible (up to the factor of $\log Q$). A bound of this shape is essentially proved by Montgomery and Vaughan \cite{MV86}, and versions have been rediscovered a number of times, for example in \cite[Lemma 5]{BDFKK}.
    \item When $k$ is even, a version of this bound with an explicit dependence on $k$ was also proved by the first author and Maynard \cite[Theorem 2]{BM22}, where it played a crucial role in improving the upper bound for the size of sets of integers without square differences.
    \item When $n = 1$ and $k = 4$, Heath-Brown \cite{H-B} proves optimal bounds, including the correct power of $\log$, for the number of `non-trivial' solutions (i.e.~ excluding the diagonal contribution from those solutions where $a_1/q_1=-a_2/q_2$ and so on.) 
\end{itemize}

To our knowledge Theorem~\ref{th-simpfrac} is the first result which addresses the case of general $n$ and general $k$, a generality which is crucial for our applications. 

A simple application of the Cauchy--Schwarz inequality together with the available bounds for even $k$ immediately implies an upper bound of the shape
\[\ll (\log Q)^{O(1)}n^{\frac{k}{2}}Q^{\frac{k}{2}}\]
for odd $k\geq 3$. (In particular, note that when $n=Q$ the bound for odd $k$ follows immediately from the bound for even $k$, which is much easier to establish.) Theorem~\ref{th-simpfrac} represents a saving of $(n/Q)^{1/2}$ over this `trivial' bound. Proving a bound of this strength is the main novel contribution of this paper, and such a bound is the driving force behind the proofs of Theorems~\ref{thm:montgomery-vaughan-odd} and \ref{thm:conjecture-for-Rkh-for-k-odd}.

We defer further discussion of Theorem~\ref{th-simpfrac}, including lower bounds and heuristic predictions, to Section~\ref{sec-problems}. We prove Theorem~\ref{th-simpfrac} (in fact a more general form that is necessary for applications) in Section~\ref{sec-fractions}. The deduction of Theorems~\ref{thm:montgomery-vaughan-odd} and \ref{thm:conjecture-for-Rkh-for-k-odd} is done in Section~\ref{sec-apps}.

\subsection*{Acknowledgements}
The first author is supported by Royal Society University Research Fellowship. The second author is supported by the National Science Foundation's Mathematical Sciences Research program through the grant DMS-2202128. We thank Tim Browning, Daniel Loughran, and Valentin Blomer for discussions and references concerning Theorem~\ref{th-simpfrac}. We thank the anonymous referees for helpful comments and suggestions.

\section{Adding an odd number of fractions}\label{sec-fractions}

The new device which allows for our main results is the following upper bound on the number of solutions to a rational equation with an odd number of summands.

\begin{theorem}\label{th-oddsum}
Let $k\geq 1$ be an integer. Let $Q_1,\ldots,Q_{2k+1}\geq 1$ and $A_1,\ldots,A_{2k+1}\subseteq [-1,1]$ be intervals of lengths $\delta_1,\ldots,\delta_{2k+1}$ respectively, where $\delta_i\geq 1/Q_i$ for $1\leq i\leq 2k+1$. 

There exists some $X\subseteq \{1,\ldots,2k+1\}$ with $\abs{X}=k+1$ such that the number of solutions to
\[\frac{a_1}{q_1}+\cdots+\frac{a_{2k+1}}{q_{2k+1}}\in \bbz\quad\textrm{ where }\quad\ogcd{a_i}{q_i}=1\textrm{ for }1\leq i\leq 2k+1,\] 
with $a_i/q_i\in A_i$ and $q_i\leq Q_i$ for $1\leq i\leq 2k+1$, is
\[\ll (\log(Q_1\cdots Q_{2k+1}))^{O(1)}\brac{\prod_{i\in X}\delta_i}Q_1\cdots Q_{2k+1}.\]
\end{theorem}

Theorem~\ref{th-simpfrac} is an immediate corollary. The proof of Theorem~\ref{th-oddsum} is quite technical, and so we will first give sketch proofs of three special cases: first when $k=3$ and all the intervals $A_i$ are identical and centred at the origin, secondly when $k=3$ without any restriction on the intervals, and finally when $k=5$, the intervals are centred on the origin, and the denominators $q_i$ have a particular arithmetic structure. These proofs offer representative illustrations of the main ideas, and so we hope will help the reader to follow the proof of Theorem~\ref{th-oddsum} itself, which is a more technical and general elaboration of the same theme.

Our technique also proves the following slight variant of Theorem \ref{th-oddsum}:
\begin{theorem}\label{th-oddsum-n-within-sets}
Let $k \geq 1$ be an integer. Let $Q_1, \ldots, Q_{2k+1} \geq 1$ and $B_i \subseteq [-Q_i,Q_i]$ for $1 \le i \le 2k+1$. There exists a partition $\{1, \ldots, 2k+1\} = X \sqcup Y$ with $|X| = k+1$ and $|Y| = k$ such that the number of solutions to
\[\frac{a_1}{q_1}+\cdots+\frac{a_{2k+1}}{q_{2k+1}}\in \bbz\quad\textrm{ where }\quad\ogcd{a_i}{q_i}=1\textrm{ for }1\leq i\leq 2k+1,\] 
with $a_i\in B_i$ and $q_i\leq Q_i$ for $1\leq i\leq 2k+1$, is
\[\ll (\log(Q_1\cdots Q_{2k+1}))^{O(1)}\brac{\prod_{i\in X}|B_i|}\brac{\prod_{j \in Y} Q_j}.\]
\end{theorem}
We omit the proof of Theorem~\ref{th-oddsum-n-within-sets}, which follows along essentially identical lines to the proof of Theorem~\ref{th-oddsum}. In fact, the special cases that we sketch when $k = 3$ or $k=5$ and all intervals are centered at the origin are also special cases of Theorem~\ref{th-oddsum-n-within-sets}.

We make essential use of the notion of \emph{relative greatest common divisors}: this is the observation that $k$-tuples $q_1,\ldots,q_k$ of natural numbers can be parameterised by $2^k$-tuples of natural numbers $g_I$, where $I$ ranges over all subsets of $\{1,\ldots,k\}$, where $g_\emptyset = 1$, with the properties that 
\[q_i=\prod_{I \ni i}g_I\textrm{ for all }i\in \{1, \ldots, k\}\]
and
\[\ogcd{g_I}{g_J}=1\textrm{ unless }I\subseteq J \textrm{ or }J\subseteq I.\]
For example, when $k=2$ this is just the fact that pairs of natural numbers $(a,b)$ are parameterised by triples $(a',b',d)$ with $\ogcd{a'}{b'}=1$, where $a=a'd$ and $b=b'd$, by taking $d=\ogcd{a}{b}$. 

The concept of relative greatest common divisors is an old one within number theory, although perhaps not as widely known as it deserves; we give a brief history, as well as details of the construction and proofs of the necessary properties, in Appendix~\ref{app1}. We recommend the reader new to relative greatest common divisors peruse this appendix before continuing with the rest of this section. Parameterisations of this kind are a useful tool when counting solutions to linear equations in the rationals, and in particular played a central role in the work of La Bret\`{e}che \cite{dlB} and Heath-Brown \cite{H-B}.

\subsection{The case $k=3$}\label{subsec:k-3-first-case}
As a warm-up, we first show that the number of solutions to 
\begin{equation}\label{eq-k3}
\frac{a_1}{q_1}+\frac{a_2}{q_2}+\frac{a_3}{q_3}=0\quad\textrm{ where }\quad\ogcd{a_i}{q_i}=1\textrm{ for }1\leq i\leq 3, 
\end{equation}
with $\abs{a_i}\leq n$, and $q_i\leq Q$, is $O_\epsilon(n^2Q^{1+\epsilon})$ for any $\epsilon>0$. 

Using relative greatest common divisors, there exist integers $g_I$ for $I\subseteq\{1,2,3\}$ such that $\ogcd{g_I}{g_J}=1$, unless either $I\subseteq J$ or $J\subseteq I$, with
\[q_1=g_{123}g_{12}g_{13}g_1\quad q_2=g_{123}g_{12}g_{23}g_2\quad q_3=g_{123}g_{13}g_{23}g_3.\]
Without loss of generality, relabelling if necessary, we can assume that $g_{23}\leq g_{12}$. Multiplying both sides of \eqref{eq-k3} by $\prod_{I\subseteq \{1,2,3\}}g_I$ we have
\[a_1g_{23}g_{2}g_3+a_2g_{13}g_1g_3+a_3g_{12}g_1g_2=0.\]
Note that $g_1$ divides the second two summands, so $g_1\mid a_1g_{23}g_2g_3$. On the other hand, $\ogcd{a_1}{g_1}=\ogcd{a_1}{q_1}=1$, and $\ogcd{g_1}{g_{23}g_2g_3}=1$ by properties of relative greatest common divisors, and hence this forces $g_1=1$. We can similarly deduce that $g_2=g_3=1$. The condition \eqref{eq-k3} has now been simplified to
\[a_1g_{23}+a_2g_{13}+a_3g_{12}=0.\]
We choose, at a cost of $Q$, the value of $q_1=g_{123}g_{12}g_{13}$. We now make use of the \emph{divisor bound} (see for example \cite[equation (2.20)]{MV06}), which says that for any $\epsilon > 0$, the number of divisors of an integer $z$ is $\ll_\epsilon z^\epsilon$. Thus, given a choice of $q_1$, we can choose individual values of $g_{123}, g_{12}, g_{13}$, inflating the number of choices made by an additional factor of $O_\epsilon(Q^\epsilon)$. 

We further fix the values of $a_1$ and $a_2$, at an additional cost of $O(n^2)$. The total number of choices made thus far is $O_\epsilon(n^2Q^{1+\epsilon})$. We have used up our budget, and therefore need to show that the remaining unknown values of $a_3$ and $g_{23}$ can be determined from the information already known. For this we note that 
\[a_1g_{23}+a_2g_{13}\equiv 0\pmod{g_{12}}.\]
Since $\ogcd{a_1}{g_{12}}=1$ (and the values of $a_1,a_2,g_{12},g_{13}$ are all known) we can deduce the value of $g_{23}\pmod{g_{12}}$. Recalling our assumption that $g_{23}\leq g_{12}$ it follows that $g_{23}$ itself is known, and now $a_3$ can be recovered, as the only remaining unknown. This concludes the proof.

We note that there is an alternative way to proceed: we could instead fix the values of $a_1,a_2,q_3$, and then recover $a_3$ and $g_{12}$, which is accomplished by the divisor bound since
\[a_3g_{12}=-a_1g_{23}-a_2g_{13}\]
is known. This method is perhaps simpler (and does not require the assumption that $g_{23}\leq g_{12}$) but the generalisation to other odd $k>3$ is (or at least seems to be) quantitatively poorer than the previous method.

\subsection{A harder case when $k=3$}\label{subsec-k3}
For our applications it is important to handle the case when the constraint $\abs{a_i}\leq n$ is replaced by $a_i/q_i\in I_i$, where $I_i\subseteq [-1,1]$ is some fixed interval. We will now explain how this generalisation is handled when $k=3$.

Let $Q_1,Q_2,Q_3\geq 1$ and $A_1,A_2,A_3\subseteq [-1,1]$ be some fixed intervals of lengths $\delta_i\gg 1/Q_i$. We will show that the number of solutions to 
\[\frac{a_1}{q_1}+\frac{a_2}{q_2}+\frac{a_3}{q_3}=0\quad\textrm{ where }\quad\ogcd{a_i}{q_i}=1\textrm{ for }1\leq i\leq 3\]
with $\frac{a_i}{q_i}\in A_i$ and $q_i\leq Q_i$ for $1\leq i\leq 3$ is $O_\epsilon(\delta_1\delta_2(Q_1Q_2Q_3)^{1+\epsilon})$ for any $\epsilon>0$, assuming $\delta_1\geq \delta_2\geq \delta_3$, say. (The previous subsection corresponds to the special case when all $Q_i=Q$ and $A_i=A$, some interval of width $O(n/Q)$ centred at $0$.)

As above, we decompose the $q_i$ into relative greatest common divisors $g_I$; recall from Section \ref{subsec:k-3-first-case} that $g_1 = g_2 = g_3 = 1$. At the cost of only $(\log Q_1Q_2Q_3)^{O(1)}$ we can assume that each $g_I\in [G_I,2G_I]$ for some known integers $G_I\geq 1$, for each $|I| \ge 2$. We can assume without loss of generality that $G_{23}\geq G_{12}$. We choose the values of $g_{123},g_{12},g_{13}$ (and thus $q_1$) at a cost of $O(G_{123}G_{12}G_{13}) = O(Q_1)$. We further choose an interval of width $O(1/Q_1)$ which contains $a_1/q_1$ and an interval of width $O(1/Q_2)$ which contains $a_2/q_2$, which costs a further $O(\delta_1\delta_2Q_1Q_2)$, since $A_i$ can be divided into $O(\delta_i Q_i)$ many sub-intervals of width $O(1/Q_i)$. Note that then, since we know an interval of width $O(1/q_1)$ which contains $a_1/q_1$, and the denominator $q_1$ is known, we also know $a_1$ (after perhaps an additional cost of $O(1)$). We have spent a total of $O_\epsilon(\delta_1\delta_2Q_1^{2+\epsilon}Q_2)$, and need to now argue that we can determine the remaining values of $g_{23},a_2,a_3$ at a cost of $O(G_{23}/G_{12})=O(Q_3/Q_1)$.

As in the previous subsection, we have the modular constraint that
\[a_1g_{23}+a_2g_{13}\equiv 0\pmod{g_{12}}.\]
This means (since $a_1,g_{12},g_{13}$ are all known, and $a_1g_{13}$ is coprime to $g_{12}$) that there is some known integer $1\leq r<g_{12}$ such that $a_2\equiv rg_{23}\pmod{g_{12}}$. That is, there must exist some integer $k$ such that $a_2=rg_{23}+kg_{12}$, whence
\[\frac{a_2}{q_2}=\frac{r}{g_{12}g_{123}}+\frac{k}{g_{23}g_{123}}.\]
(Note that since $\ogcd{a_2}{g_{23}}=1$ we must have $\ogcd{k}{g_{23}}=1$.) Since $a_2/q_2$ lies inside some known interval of length $O(1/Q_2)$ we deduce that $k/g_{23}g_{123}$ lies inside some known interval of length $O(1/Q_2)$, and hence $k/g_{23}$ lies inside some known interval of length $O(1/G_{12}G_{23})$. We now observe that any two distinct fractions $\frac{b_1}{r_1}\neq \frac{b_2}{r_2}$ with denominators in $[G_{23},2G_{23}]$ are separated by $\gg 1/G_{23}^2$, since
\[\abs{\frac{b_1}{r_1}-\frac{b_2}{r_2}}=\abs{\frac{b_1r_2-b_2r_1}{r_1r_2}}\geq \frac{1}{r_1r_2}\gg \frac{1}{G_{23}^2}.\]
Since $G_{12}\leq G_{23}$ by assumption, therefore, there are only $O(G_{23}/G_{12})$ many possible choices for $k/g_{23}$. Fixing such a choice fixes both $g_{23}$ and $a_2$, and then $a_3$ is known, and the solution is determined.

\subsection{A special case when $k=5$}
We now sketch a proof that the number of solutions to 
\begin{equation}\label{eq-k5}
\frac{a_1}{q_1}+\frac{a_2}{q_2}+\frac{a_3}{q_3}+\frac{a_4}{q_4}+\frac{a_5}{q_5}=0\quad\textrm{ where }\quad \ogcd{a_i}{q_i}=1\textrm{ for }1\leq i\leq 5, 
\end{equation}
with $\abs{a_i}\leq \min(n,q_i)$ and $q_i\in [Q,2Q]$, is $O_\epsilon(n^3Q^{2+\epsilon})$ for any $\epsilon>0$. As before, we will use relative greatest common divisors. Here we consider a special case: we restrict ourselves to bounding the number of solutions under the additional assumption that $g_I=1$ unless $\abs{I}=2$. Certainly not all solutions satisfy this assumption, but restricting the number of $g_I$ that need to be considered allows for a clearer presentation of the essential ideas.

Clearing denominators in \eqref{eq-k5} we need to bound the number of $a_i$ for $1\leq i\leq 5$ and $g_I$ for $I\subseteq \{1,2,3,4,5\}$ with $\abs{I}=2$, such that
\begin{equation}\label{eq-k5'}
a_1\prod_{I\not\ni 1}g_I+a_2\prod_{I\not\ni 2}g_I+a_3\prod_{I\not\ni 3}g_I+a_4\prod_{I\not\ni 4}g_I+a_5\prod_{I\not\ni 5}g_I=0.
\end{equation}
Losing only a factor of $(\log Q)^{O(1)}$ in the final bound, it suffices to count those solutions where $g_I\in [G_I,2G_I]$ for some fixed integers $G_I\geq 1$, for all $I$. (Notice that since $q_i=\prod_{I\ni i}g_I$ we have $Q\asymp \prod_{I\ni i}G_I$ for all $i$.)

Let $X\sqcup Y=\{1,2,3,4,5\}$ with $\abs{X}=3$ and $\abs{Y}=2$ be a partition chosen to minimise 
\[\prod_{I\subseteq X}G_I\prod_{J\subseteq Y}G_J.\]
Note that this implies that, for all $i\in X$,
\begin{equation}\label{eq-k5''}
\prod_{\substack{y \in Y}}G_{iy}\geq \prod_{\substack{x \in X}}G_{ix},
\end{equation}
or else the minimised product can be decreased by replacing $X$ with $Y\cup \{i\}$ and $Y$ with $X\backslash \{i\}$. Relabelling if necessary, we can assume that $X=\{1,2,3\}$ and $Y=\{4,5\}$. 

We begin by fixing the values of $a_1,a_2,a_3$, and $g_{12},g_{13},g_{14},g_{15},g_{23},g_{45}$ at a total cost of 
\[\ll n^3G_{12}G_{13}G_{14}G_{15}G_{23}G_{45}\ll n^3QG_{23}G_{45}.\]
It remains to find the values of $a_4,a_5,g_{24},g_{25},g_{34},g_{35}$. We will first choose the values of $g_{24}$ and $g_{25}$. Reducing \eqref{eq-k5'} modulo $g_{12}$ yields
\[g_{34}g_{35}g_{45}\brac{a_1g_{23}g_{24}g_{25}+a_2g_{13}g_{14}g_{15}}\equiv 0\Mod{g_{12}}.\]
Since $\ogcd{g_{34}g_{35}g_{45}}{g_{12}}=\ogcd{a_1g_{23}}{g_{12}}=1$, and the values of $a_1,a_2,g_{23},g_{13},g_{14},g_{15}$ are all known, it follows that the value of
\[g_{24}g_{25}\Mod{g_{12}}\]
is known. Since $g_{24}g_{25}\in [G_{24}G_{25},4G_{24}G_{25}]$, the value of $g_{24}g_{25}$ is determined after paying a cost of
\[\ll 1+ \frac{G_{24}G_{25}}{g_{12}}\ll 1+\frac{G_{24}G_{25}}{G_{12}}\ll \frac{G_{24}G_{25}}{G_{12}},\]
since, by \eqref{eq-k5''},
\[G_{24}G_{25}\geq G_{12}G_{23}\geq G_{12}.\]
The values of $g_{24}$ and $g_{25}$ can be recovered via the divisor bound, at an additional cost of $O_\epsilon(Q^\epsilon)$. 

Similarly, reducing \eqref{eq-k5'} modulo $g_{13}g_{23}$ and dividing by $g_{45}$ (which must be relatively prime to $g_{13}g_{23}$),
\[g_{34}g_{35}\brac{a_1g_{23}g_{24}g_{25}+a_2g_{13}g_{14}g_{15}}+a_3g_{12}g_{14}g_{15}g_{24}g_{25} \equiv 0 \Mod{g_{13}g_{23}}.\]
Since 
\[\ogcd{a_3g_{12}g_{14}g_{15}g_{24}g_{25}}{g_{13}g_{23}}=\ogcd{g_{34}g_{35}}{g_{13}g_{23}}=1\]
(because $\ogcd{a_3}{q_3}=1$ and $\ogcd{g_I}{g_J}=1$ whenever $\abs{I}=\abs{J}=2$ and $I\neq J$),
and since
\[\ogcd{a_1g_{23}g_{24}g_{25} + a_2g_{13}g_{14}g_{15}}{g_{13}g_{23}}=1\]
(because by the same logic as the previous, the first summand is relatively prime to $g_{13}$ and the second is relatively prime to $g_{23}$), the value of 
\[g_{34}g_{35}\Mod{g_{13}g_{23}}\]
is known. As above, since by \eqref{eq-k5''} we have $G_{34}G_{35}\geq G_{13}G_{23}$, the values of $g_{34}$ and $g_{35}$ can now be fixed after paying
\[\ll_\epsilon Q^\epsilon \frac{G_{34}G_{35}}{G_{13}G_{23}}.\]
The values of all $g_I$ are now known, and it suffices to fix the values of $a_4$ and $a_5$. The value of $a_4$ is (again consulting \eqref{eq-k5'}) known modulo $g_{14}g_{24}g_{34}$, and so since it suffices to fix the value of $a_4\Mod{q_4}$ we can recover $a_4$ after paying
\[\ll \frac{Q}{G_{14}G_{24}G_{34}}.\]
The value of $a_5$ is now fixed, as it is the only remaining unknown in \eqref{eq-k5'}. All the variables are now chosen, and it remains to total up our costs. We have spent a total of
\[\ll_\epsilon Q^\epsilon n^3\brac{QG_{23}G_{45}}\brac{\frac{G_{24}G_{25}}{G_{12}}}\brac{\frac{G_{34}G_{35}}{G_{13}G_{23}}}\brac{\frac{Q}{G_{14}G_{24}G_{34}}}=n^3Q^{2+\epsilon}\frac{G_{25}G_{35}G_{45}}{G_{12}G_{13}G_{14}}.\]
Recalling that $G_{25}G_{35}G_{45}\ll Q/G_{15}$, and $Q\ll  G_{12}G_{13}G_{14}G_{15}$, this bound is
\[\ll n^3Q^{3+\epsilon}\frac{1}{G_{12}G_{13}G_{14}G_{15}}\ll n^3Q^{2+\epsilon}\]
as required.

\subsection{Proof of Theorem~\ref{th-oddsum}}
We will now prove Theorem~\ref{th-oddsum} in full. The difficulties compared to the simplified sketches above are (i) to generalise that machinery to handle arbitrary odd $k\geq 3$, (ii) to handle the potentially non-trivial contribution from those $g_I$ with $\lvert I\rvert\geq 2$, (iii) to take more care with bounding divisors to produce factors of $(\log Q)^{O(1)}$ rather than $Q^\epsilon$, and (iv) to allow for constraints of the shape $a_i/q_i\in I_i$ as in Section~\ref{subsec-k3}. It is a non-trivial task to handle these issues, but nonetheless we view them as largely technical, and believe that the presentation of the simplified cases above captures the essential spirit of the proof.

All implicit constants in this proof may depend on $k$ in some unspecified fashion. For convenience we write $Q=Q_1\cdots Q_{2k+1}$. In this proof we use $\tau_r(n)$ to denote the $r$-fold divisor function (that is, the number of representations $n=d_1\cdots d_r$), and recall the elementary estimate (used without further mention) that, for any constant $\ell>0$ and integer $r\geq 1$,
\begin{equation}\label{eq-divbound}
\sum_{n\leq N}\tau_r(n)^\ell \leq (\log N)^{O_{\ell,r}(1)}N.
\end{equation}
See for example \cite[equation (2.31)]{MV06}, noting that $\tau_r(n) \le \tau(n)^r$ for all $r \ge 1$.

It suffices to bound the number of solutions to 
\begin{equation}\label{eq-fracsum}
\frac{a_1}{q_1}+\cdots+\frac{a_{2k+1}}{q_{2k+1}}=n
\end{equation}
for any fixed integer $n$ (since there are trivially $O(1)$ possible integers that need to be considered). Let $g_I$, as $I$ ranges over subsets of $\{1,\ldots,2k+1\}$, be the relative greatest common divisor parameterisation of $q_i$. At the acceptable cost of an additional factor of $(\log Q)^{O(1)}$, it suffices to prove the target bound for the number of solutions to \eqref{eq-fracsum} under the additional assumption that $g_I\in [G_I,2G_I]$ for some fixed integers $G_I$. Let $\{1,\ldots,2k+1\}=X\sqcup Y$ be a partition with $\abs{X}=k+1$ and $\abs{Y}=k$ such that 
\begin{equation}\label{eq-minim}
\prod_{I\subseteq X}G_I \prod_{J\subseteq Y}G_J
\end{equation}
is minimised. Relabelling if necessary, we will assume that $X=\{1,\ldots,k+1\}$ and $Y=\{k+2,\ldots,2k+1\}$. (Note that in particular this choice of partition depends on the particular values of the $G_I$, and hence perhaps different partitions $X\sqcup Y$ may be used to bound different classes of solutions, but this clearly does not affect the validity of Theorem~\ref{th-oddsum}, since we can simply take a maximum of the upper bound over all possible partitions.)

Before proceeding with the main line of the proof we pause to record the useful inequality that, for all $1\leq i\leq k+1$,
\begin{equation}\label{eq-usefulbound}
\prod_{\substack{I\ni i\\I\subseteq\{i\}\cup Y}}G_I\geq \prod_{\substack{J\ni i\\ J\subseteq \{1,\ldots,i\}}}G_{J}.
\end{equation}
Indeed, \eqref{eq-usefulbound} follows immediately from the stronger inequality
\[\prod_{\substack{I\ni i \\I\subseteq \{i\}\cup Y}}G_I\geq \prod_{\substack{J\ni i\\ J\subseteq X}}G_{J},\]
which in turn follows by noting that if we replace $Y$ by $X\backslash \{i\}$ and $X$ by $Y\cup \{i\}$ then the quantity in \eqref{eq-minim} changes by a multiplicative factor of
\[\frac{\prod_{\substack{I\ni i\\I\subseteq \{i\}\cup Y}}G_I}{\prod_{\substack{J\ni i\\ J\subseteq X}}G_{J}},\]
and so by the assumed minimality of \eqref{eq-minim} this quantity is $\geq 1$ as required.

Upon multiplying both sides by $\prod_I g_I$, equation \eqref{eq-fracsum} is equivalent to
\begin{equation}\label{eq-simpsum}
a_1\prod_{I\not \ni 1}g_I+\cdots +a_{2k+1}\prod_{I\not \ni 2k+1}g_I= n\prod_I g_I.
\end{equation}
Note that $g_{\{i\}}$ divides the right-hand side and every summand on the left-hand side except $a_i\prod_{I\not\ni i}g_I$, and therefore $g_{\{i\}}\mid a_i\prod_{I\not\ni i}g_I$. Since $\ogcd{a_i}{q_i}=1$ and $\ogcd{g_{\{i\}}}{g_I}=1$ for $i\not\in I$ it follows that $g_{I}=1$ whenever $\abs{I}=1$. 

The proof will proceed along the following steps:
\begin{itemize}
\item \textbf{Step one.} For each $i \in X$ we fix $B_i$, a sub-interval of $A_i$ of length $< 1/Q_i$ in which $\frac{a_i}{q_i}$ will lie. We further fix all $g_I$ except those where $|I \cap X| = 1$. The total cost to fix these terms is
\begin{equation*}
\ll \prod_{i \in X} \delta_iQ_i\prod_{\substack{I \\ |I \cap X|\neq 1}} G_I.
\end{equation*}
\item \textbf{Step two.} We now carefully fix (and count possibilities for) the values of the $g_I$ with $|I \cap X| = 1$; these values are heavily constrained by the values which have already been fixed (i.e.~ the constraints from step one on $a_i/q_i$ for $i \in X$, and $g_I$ for all $I$ with $|I \cap X| \ne 1$). This step is by far the most complicated, largely because in order to only lose a factor of $(\log Q)^{O(1)}$, we must be careful to only use the divisor bound on average. The additional cost to fix these terms is
\begin{equation*}
\ll (\log Q)^{O(1)} \frac{\prod_{\abs{I\cap X}=1}G_I}{\prod_{J\subseteq X}G_J}.
\end{equation*}
\item \textbf{Step three.} All denominators have been fixed, which (since $a_i/q_i$ is known to lie in some fixed interval of width $<1/Q_i\leq 1/q_i$ for $i\in X$) also determines $a_i$ for $i \in X$. It remains only to count the possibilities for $a_i$ for all $i \in Y$. These values are almost fully determined by what has already been fixed; the number of possibilities for these $a_i$, given the information known at this point, is 
\begin{equation*}
\ll \prod_{\substack{I\\ \abs{I\cap Y}\geq 2}}G_I^{\abs{I\cap Y}-1}.
\end{equation*}
\end{itemize}

Multiplying these three contributions together and simplifying, we get an upper bound of
\begin{equation*}
\ll (\log Q)^{O(1)} \prod_{i \in X} \delta_i Q_i \prod_{\substack{I \\ I \not\subseteq X}} G_I^{|I \cap Y|}.
\end{equation*}
The proof is complete after noting that $Q_i\asymp \prod_{I \ni i}G_I$, and hence 
\[\prod_{j\in Y}Q_j\asymp \prod_I G_I^{\abs{I\cap Y}}.\]

It remains to prove steps one through three.

\subsubsection{Step one:}
This step is immediate by assumption on the $\frac{a_i}{q_i}$, and recalling that $g_I \in [G_I, 2G_I]$ for all $I$.

\subsubsection{Step two:}
We have now determined an interval $B_i$ of width $<1/Q_i$ containing $a_i/q_i$ for $i\in X$, as well as all $g_I$ except those where $\abs{I\cap X}=1$; our next goal is to fix the values of $g_I$ with $\abs{I\cap X}=1$. 

For brevity, define
\begin{equation}\label{eq-hdef}
h_i = \prod_{\substack{I \ni i\\ I\subseteq \{i\} \cup Y}}g_I
\end{equation}
for $i\in X$, and similarly $H_i= \prod_{\substack{I \ni i\\ I\subseteq \{i\} \cup Y}}G_I$, so that $h_i\asymp H_i$. We need to both determine the value of each $h_i$ and also fix a decomposition of $h_i$ into $\leq 2^{k}$ many factors, which will fix the value of each $g_I$ in the product. 

Let $P(h_1,\ldots,h_{k+1},a_1, \dots, a_{k+1})$ be the event that there exist some $a_{k+2}, \ldots, a_{2k+1}$ and decomposition of each $h_i$ into factors as in \eqref{eq-hdef} which (in conjunction with the already fixed $g_I$ for $\abs{I\cap X}\neq 1$) form a solution to \eqref{eq-simpsum} and satisfy the constraint that $a_i/q_i \in B_i$ for all $i \in X$. The total number of possibilities for all remaining $g_I$ can then be bounded above by
\[\sum_{h_1,\ldots,h_{k+1},a_1,\ldots,a_{k+1}}1_{P(h_1,\ldots,a_{k+1})}\brac{\prod_i 1_{h_i \asymp H_i}}\tau_{2^k}(h_1)\cdots \tau_{2^k}(h_{k+1}).\]
Using the elementary fact that $x_1\cdots x_{k+1}\leq x_1^{k+1}+\cdots+x_{k+1}^{k+1}$ it therefore suffices to show that
\begin{equation}\label{eq-hdivsum}
\sum_{h_1,\ldots,h_{k+1},a_1,\ldots,a_{k+1}}1_{P(h_1,\ldots,a_{k+1})}\brac{\prod_i 1_{h_i \asymp H_i}}\tau_{2^k}(h_1)^{k+1}
\end{equation}
is 
\[\ll (\log Q)^{O(1)}\prod_{1\leq i\leq k+1}\frac{\prod_{\substack{I \ni i\\ I\subseteq \{i\}\cup Y}}G_I}{\prod_{\substack{J \ni i\\ J\subseteq \{1,\ldots,i\}}}G_J}.\]

We now note that the condition $P(h_1,\ldots,h_{k+1}, a_1, \dots, a_{k+1})$ implies a certain collection of cascading restrictions. Fix some $1\leq i\leq k+1$. Multiplying equation \eqref{eq-fracsum} by $\prod_{I \not\subseteq \{1, \dots, i-1\}} g_I$ (recalling that $q_i=\prod_{i\in I}g_I$) yields
\begin{equation} \label{eq:general-sum-multiplied-up-to-i}\sum_{1\leq j<i}a_j\frac{\prod_{I \not\subseteq \{1, \dots, i-1\}} g_I}{\prod_{j\in I}g_I}+a_i\prod_{\substack{I \not\ni i\\ I\not\subseteq \{1, \dots, i-1\}}} g_I+\sum_{j>i}a_j\prod_{\substack{I \not\ni j\\ I\not\subseteq \{1, \dots, i-1\}}} g_I=n\prod_{I\not\subseteq \{1,\ldots,i-1\}}g_I.
\end{equation}
In particular, the first sum is an integer, so defining
\begin{equation}\label{eq:def-of-b_i}
    b_i=\sum_{1\leq j<i}a_j \frac{\prod_{\substack{I \not\ni j \\ I \not\subseteq \{1, \dots, i-1\}\\ I\not\subseteq \{i,\ldots,2k+1\}}}g_I}{\prod_{j \in I \subseteq \{1, \dots, i-1\}} g_I },
\end{equation}
we have that
\[\brac{\prod_{I\subseteq \{i,\ldots,2k+1\}}g_I}b_i=\sum_{1\leq j<i}a_j\frac{\prod_{I \not\subseteq \{1, \dots, i-1\}} g_I}{\prod_{I \ni j}g_I}\]
must be an integer. We claim further that $b_i$ itself is an integer. To see this, let $b_i = \frac{x_i}{y_i}$ be a fraction in lowest terms. The denominator $y_i$ must satisfy $y_i \mid\prod_{I \subseteq \{1, \ldots, i-1\}} g_I$, which is divisible by the denominator of each term in \eqref{eq:def-of-b_i}. Note that if $I \subseteq \{1, \ldots,i-1\}$ and $J \subseteq \{i, \ldots, 2k+1\}$, then $\ogcd{g_I}{g_J} = 1$. Thus
\[\bracgcd{y_i}{\prod_{I \subseteq \{i,\ldots,2k+1\}}g_I} \mid \bracgcd{\prod_{I \subseteq \{1, \ldots, i-1\}}g_I}{\prod_{I \subseteq \{i,\ldots,2k+1\}} g_I } = 1.  \]
But since $\frac{x_i}{y_i} \prod_{I \subseteq \{i, \ldots, 2k+1\}} g_I$ is an integer and $y_i$ is relatively prime to the product, we must have $y_i |x_i$ and thus $b_i$ is also an integer.

It is important to observe that $b_i$ depends only on $a_j$ and $h_j$ for $1\leq j<i$, as well as the values of $g_I$ with $\abs{I\cap X}\neq 1$ (which are all known).

Reducing equation \eqref{eq:general-sum-multiplied-up-to-i} modulo $\prod_{\substack{J \ni i \\ J \subseteq \{1, \dots, i\}}}g_J$ implies that
\[b_i\prod_{I\subseteq \{i,\ldots,2k+1\}}g_I\equiv -a_i \prod_{\substack{I \not\ni i \\ I \not\subseteq\{1, \dots, i-1\}}} g_I \pmod{\prod_{\substack{J \ni i \\ J \subseteq \{1,\ldots,i\}}} g_J},\]
that is,
\[b_i\prod_{j \ge i} h_j \prod_{\substack{I \subseteq\{i, \ldots, 2k+1\} \\ |I \cap X| \ne 1}} g_I \equiv -a_i \prod_{j \ne i}h_j \prod_{\substack{I \not\ni i \\ I \not\subseteq\{1,\ldots,i-1\}\\ \abs{I\cap X}\ne 1}}g_I \pmod{\prod_{\substack{J \ni i \\ J \subseteq \{1,\ldots,i\}}} g_J}.\]
Note that every term on each of the left and right sides is relatively prime to the modulus (where for $b_i$ we can conclude that it is relatively prime to the modulus because all the remaining terms are). Thus repeated terms may be canceled from the two sides to get that
\[b_i h_i \prod_{\substack{I \subseteq\{i, \ldots, 2k+1\} \\ |I \cap X| \ne 1}} g_I \equiv -a_i \prod_{j < i}h_j \prod_{\substack{I \not\ni i \\ I \not\subseteq\{1,\ldots,i-1\}\\ \abs{I\cap X}\ne 1}}g_I \pmod{\prod_{\substack{J \ni i \\ J \subseteq \{1,\ldots,i\}}} g_J}.\]

We conclude that there exists some congruence class $r_i \mod \prod_{\substack{J \ni i \\ J \subseteq \{1, \ldots, i\}}} g_J$, depending only on $g_I$ with $|I \cap X| \ne 1$ and on $a_j$ and $h_j$ for $j < i$, such that
\[h_ir_i \equiv a_i \mod \prod_{\substack{J \ni i \\ J \subseteq \{1, \ldots, i\}}} g_J.\]

At the same time, for $1 \le i \le k+1$, we have $q_i = h_i \prod_{\substack{I \ni i \\ |I \cap X| \ge 2}} g_I$. Thus $\frac{a_i}{q_i} \in B_i$ if and only if
\[ \frac{a_i}{h_i} \in \left(\prod_{\substack{I \ni i \\ |I \cap X| \ge 2}} g_I\right) B_i.\]
Let $C_i$ denote the interval $\left(\prod_{\substack{I \ni i \\ |I \cap X| \ge 2}} g_I\right) B_i$, and note that (recalling $B_i$ is an interval of length $O(1/Q_i)$), the interval $C_i$ has length $\ll O(1/H_i)$. 
The left-hand side of \eqref{eq-hdivsum} can therefore be bounded above by 
\[\sum_{h_1,\ldots,h_{k+1},a_1,\ldots,a_{k+1}}\tau_{2^k}(h_1)^{k+1}\prod_{1\leq i\leq k+1} 1_{h_i \asymp H_i}1_{a_i\equiv r_ih_i\bmod{f_i}}1_{a_i/h_i \in C_i}1_{a_i/q_i\in B_i}\]
(where $f_i=\prod_{i\in J\subseteq \{1,\ldots,i\}}g_J$ and the $r_j$ are integers that depend on $a_j$ and $h_j$ for $1\leq j<i$ in the manner specified above). 

We can now evaluate the sum over $a_j,h_j$ for $2\leq j\leq k+1$. We estimate these in reverse order, beginning with $j=k+1$, so that in our argument we can assume that the values of $a_i$ and $h_i$ are known for $1\leq i<j$, and hence $r_j,f_j$ are also known. Since $a_j \equiv r_j h_j \pmod{f_j}$, there exists some $\ell_j$ such that 
\[\frac{a_j}{h_jf_j} = \frac{r_j}{f_j} + \frac{\ell_j}{h_j}.\]
Since $\frac{a_j}{h_j} \in C_j$, $\frac{\ell_j}{h_j}$ lies in a known interval of length $O(1/H_jf_j)$. Since distinct fractions with denominators of size $\asymp H_j$ are separated by at least $\gg 1/H_j^2$, there are at most $O(H_j/f_j)$ acceptable choices of $\frac{\ell_j}{h_j}$ (and hence also of $h_j$), provided $f_j \ll H_j$. This is true by \eqref{eq-usefulbound}, since 
\[f_j \ll \prod_{\substack{J\ni j \\ J\subseteq \{1,\ldots,j\}}}G_J\leq \prod_{\substack{I\ni j\\ I\subseteq \{j\}\cup Y}}G_I=H_j.\]
Note that once each $h_j$ is fixed, $q_j$ is fixed as well, which means that $a_j$ is completely determined.

It follows that the left-hand side of \eqref{eq-hdivsum} is 
\[\ll \prod_{2\leq j\leq k+1}\brac{\frac{H_j}{\prod_{\substack{J\ni j \\ J\subseteq \{1,\ldots,j\}}}G_J}}\sum_{h_1\asymp H_1}\tau_{2^{k}}(h_1)^{k+1}\sum_{a_1} 1_{a_1/h_1 \in C_1}1_{a_1/q_1\in B_1}.\]
Since $C_1$ is an interval of length $\ll O(1/H_1)$, for any $h_1 \asymp H_1$ the number of integers $a_1$ such that $a_1/h_1 \in C_1$ is $O(1)$, and thus the inside sum over $a_1$ is $O(1)$.
The sum over $h_1$ is $\ll (\log Q)^{O(1)}H_1$ by \eqref{eq-divbound} and hence the left-hand side of \eqref{eq-hdivsum} is 
\[\ll (\log Q)^{O(1)}\prod_{1\leq j\leq k+1}H_j\prod_{2\leq j\leq k+1}\brac{\frac{1}{\prod_{\substack{J\ni j \\ J\subseteq \{1,\ldots,j\}}}G_J}}.\]
Recalling $H_j= \prod_{\substack{I \ni j\\ I\subseteq \{j\} \cup Y}}G_I$, this is
\[\ll (\log Q)^{O(1)}\frac{\prod_{\abs{I\cap X}=1}G_I}{\prod_{J\subseteq X}G_J}\]
as required.

\subsubsection{Step three:}
We have now fixed the values of all $a_i$ with $i \in X$ and all $g_I$. It remains to fix the values of $a_{k+2}, \dots, a_{2k+1}$, which we will do inductively. For this we note that, with $k+2 \le j \le 2k+1$, equation \eqref{eq-fracsum} implies that
\begin{multline*}
    a_1 \frac{\prod_{\substack{I \not\ni 1 \\ I \not\subseteq \{1, \dots, j-1\}}} g_I }{\prod_{1 \in I \subseteq \{1, \dots, j-1\}}g_I} + \cdots + a_{j-1} \frac{\prod_{\substack{I \not\ni j-1 \\ I \not\subseteq \{1, \dots, j-1\}}}g_I}{\prod_{j-1 \in I \subseteq \{1, \dots, j-1\}}g_I} \\
    + a_j \prod_{\substack{I \not\ni j \\ I \not\subseteq \{1, \dots, j-1\}}} g_I + \cdots + a_{2k+1} \prod_{\substack{I \not\ni 2k+1 \\ I \not \subseteq \{1, \dots, j-1\}}} g_I = n \prod_{I \not\subseteq\{1 \dots, j-1\}} g_I,
\end{multline*}
on multiplying by $\prod_{I \not\subseteq\{1 \dots, j-1\}} g_I$; note that the sum of the first $j-1$ terms must be an integer.
Reducing modulo $\prod_{\substack{J \ni j \\ J \cap \{j+1, \dots, 2k+1\} = \emptyset}} g_J$ implies that
\begin{multline*}
    a_1 \frac{\prod_{\substack{I \not\ni 1 \\ I \not\subseteq \{1, \dots, j-1\}}} g_I }{\prod_{1 \in I \subseteq \{1, \dots, j-1\}} g_I } + \cdots + a_{j-1} \frac{\prod_{\substack{I \not\ni j-1 \\ I \not\subseteq \{1, \dots, j-1\}}} g_I }{\prod_{j-1 \in I \subseteq \{1, \dots, j-1\}} g_I } \\
    \equiv -a_j \prod_{\substack{I \not\ni j \\ I \not\subseteq \{1, \dots, j-1\}}} g_I \pmod{\prod_{\substack{J \ni j\\ J\cap \{j+1,\ldots,2k+1\}= \emptyset}}g_J}.
\end{multline*}

We claim that the coefficient of $a_j$ is relatively prime to the modulus. Indeed, if $I$ is any set satisfying $I \not\ni j$ and $I \not\subseteq \{1, \dots, j-1\}$ (or equivalently $I \cap \{j+1, \dots, 2k+1\} \neq \emptyset$), and $J$ is any set satisfying $J \ni j$ and $J \cap \{j+1, \dots, 2k+1\} = \emptyset$, then $I \not\subseteq J$ (since $I$ contains some element of $\{j+1, \dots, 2k+1\}$ and $J$ does not) and $J \not\subseteq I$ (since $J$ contains $j$ and $I$ does not). Thus for any $I$ and $J$ satisfying these conditions, $\ogcd{g_I}{g_J} = 1$, and in turn
\begin{equation*}
    \bracgcd{\prod_{\substack{I \not\ni j \\ I \not\subseteq \{1, \dots, j-1\}}} g_I}{\prod_{\substack{J \ni j\\ J\cap \{j+1,\ldots,2k+1\}= \emptyset}}g_J} = 1.
\end{equation*}

It follows that, once the $a_i$ are known for $i < j$, $a_j$ is determined modulo $\prod_{\substack{J \ni j \\ J \cap \{j+1, \ldots,2k+1\}=\emptyset}}g_J$. In this case the cost to determine $a_j$ is 
\begin{equation*}
    \ll \prod_{\substack{I \ni j \\ I \cap \{j+1, \ldots, 2k+1\}\neq \emptyset}} G_I.
\end{equation*}
Multiplying this for all $j \in Y$, the total cost to determine all remaining $a_i$ is therefore
\begin{equation*}
    \ll \prod_{\substack{I \\ |I \cap Y| \ge 2}} G_I^{|I \cap Y| - 1}.
\end{equation*}
This concludes the proof of step three, and hence the proof of Theorem~\ref{th-oddsum}.

\section{Odd moments}\label{sec-apps}
We will now use Theorem~\ref{th-oddsum} to prove Theorems \ref{thm:montgomery-vaughan-odd} and \ref{thm:conjecture-for-Rkh-for-k-odd}. These proofs follow along similar lines to those from \cite{MV86} and \cite{MS04} respectively, with several technical refinements; the significant difference is our use of Theorem~\ref{th-oddsum}. 

\subsection{Proof of Theorem \ref{thm:montgomery-vaughan-odd}}
We introduce some definitions now to make the statements clearer. By \cite[Lemma 2]{MV86} we can write $M_k(q,h) =  q(\frac{\phi(q)}{q})^k V_{k}(q,h)$, where
\begin{align}
V_k(q,h) &:= \subsum{q_1, \dots, q_k \\ 1 < q_i|q} \left(\prod_{i=1}^k \frac{\mu(q_i)}{\phi(q_i)}\right)\subsum{a_1, \dots, a_k \\ 1 \le a_i \le q_i \\ \ogcd{a_i}{q_i} = 1 \\ \sum a_i/q_i \in \mathbb Z} \sum_{1 \le d_1, \ldots, d_k \le h} e\left(\sum_{i=1}^k \frac{d_ia_i}{q_i}\right) \label{eq-defn-of-Vkqh} \\
&= \subsum{q_1, \dots, q_k \\ 1 < q_i|q} \left(\prod_{i=1}^k \frac{\mu(q_i)}{\phi(q_i)}\right)\subsum{a_1, \dots, a_k \\ 1 \le a_i \le q_i \\ \ogcd{a_i}{q_i} = 1 \\ \sum a_i/q_i \in \mathbb Z} E\left(\frac{a_1}{q_1}\right) \cdots E\left(\frac{a_k}{q_k}\right),\nonumber
\end{align}
where $E(x) := \sum_{m=1}^h e(mx)$, with $e(u) = e^{2\pi i u}$. We note that $\lvert E(x)\rvert \leq F(x)$ for all $x\in \mathbb{R}$, with \begin{equation*}
F(x) := \min\{h,\|x\|^{-1}\},
\end{equation*}
where $\norm{x}$ denotes the distance from $x$ to the nearest integer. The overall strategy in proving Theorem \ref{thm:montgomery-vaughan-odd} is as follows:
\begin{enumerate}
\item First note that $M_k(q,h)$ is multiplicative in $q$, and hence we can write $q=q_Sq_R$, where $q_S$ is only divisible by primes $\leq h^k$ and $q_R$ is only divisible by primes $>h^k$, and use different methods to bound $M_k(q_R,h)$ and $M_k(q_S,h)$ separately.
\item The `smooth' component $q_S$ is controlled by Theorem~\ref{th-oddsum} (after converting $M_k$ to $V_k$ as above and applying the triangle inequality), noting that the $F(a_i/q_i)$ is essentially $h$ multiplied by an indicator function restricting $a_i/q_i$ to an interval of width $\ll 1/h$. Some care must be taken to ensure that the logarithmic term is logarithmic only in $h$, however, since naively the denominators considered may range up to $\prod_{p\leq h^k}p\approx e^{h^k}$. To this end, we follow \cite{MV86} and employ a further separation device to factorise each denominator into a small part and a large part, and apply Theorem~\ref{th-oddsum} to only the contribution from small denominators.
\item The `rough' component $q_R$ is handled by a combinatorial device, in a similar fashion to \cite{MV86}, although for our desired level of precision the argument is slightly more involved. 
\end{enumerate}

\subsubsection{The smooth component} We first explain how Theorem~\ref{th-oddsum} is used to bound $M_k(q,h)$ when $q$ is `smooth' (i.e.~ has only small prime divisors). Theorem~\ref{th-oddsum} is employed via the following technical lemma.
\begin{lemma}\label{lem:bounds-on-s-i-sums-with-alpha-shifts}
Let $h \ge 1$ be an integer, let $k \ge 3$ be an odd integer, and fix any rational $\alpha_1, \dots, \alpha_k \in [0,1)$. Suppose $2\leq y\leq h$. For each $i$, let $\ell_i = 1$ if $\alpha_i = 0$ and let $\ell_i = h$ otherwise. Then
\begin{equation*}
    \sum_{\substack{s_1, \dots, s_k \\ \ell_i < s_i < y}} \frac{\mu(s_i)^2}{\phi(s_i)} \sum_{\substack{a_1, \dots, a_k \\ 1 \le a_i \le s_i \\ \ogcd{a_i}{s_i} = 1 \\ \sum_i a_i/s_i \in \mathbb Z}} \prod_{i=1}^k F\left(\frac{a_i}{s_i} + \alpha_i\right) \ll (\log y)^{O(1)}h^{\tfrac{k-1}{2}}.
\end{equation*}
\end{lemma}
\begin{proof}
The left-hand side can be bounded above by
\begin{equation}\label{eq-lem1}
\ll (\log y)^k \sum_{1 < s_1, \dots, s_k < y} (s_1 \cdots s_k)^{-1} \sum_{\substack{a_1, \dots, a_k \\ 1 \le a_i \le s_i \\ \ogcd{a_i}{s_i} = 1 \\ \sum_i a_i/s_i \in \mathbb Z}} \prod_{i=1}^k F\left(\frac{a_i}{s_i} + \alpha_i\right),
\end{equation}
since $\frac{s}{\phi(s)} \ll \log y$ for any $s \le y$. For $1 \le a_i \le s_i$, we have
\begin{equation*}
F\left(\frac{a_i}{s_i} +\alpha_i \right) = \min\left\{h, \left\|\frac{a_i}{s_i}+\alpha_i\right\|^{-1}\right\},
\end{equation*}
so that $\frac 1{s_i} F\left(\frac{a_i}{s_i}+\alpha_i\right) \le 1/n_i$, where
\begin{equation*}
n_i = \mathrm{max}\left\{\frac{s_i}{h}, \lfloor |\tilde{a_i} + \alpha_i s_i|\rfloor \right\},
\end{equation*}
and $\tilde{a_i}$ is the representative of $a_i$ mod $s_i$ which minimises $|\tilde{a_i} + \alpha_i s_i|$ (so that $|\tilde{a_i} + \alpha_i s_i|\leq s_i/2$). 

The possible values of $n_i$ are thus the set $\left\{\frac{s_i}{h}\right\} \cup \left(\mathbb Z \cap \left[\frac{s_i}{h},\frac{s_i}{2}\right]\right)$. 
Note that if $s_i/h < 1$, then $1 < s_i < h$, so $\alpha_i = 0$ by assumption. In this case there is no contribution when $n_i = s_i/h$, since that implies $|\tilde{a_i}| \le s_i/h < 1$, and thus $\tilde{a_i} = 0$, which contradicts the requirement that $\ogcd{a_i}{s_i} = 1$. In particular, this means that $n_i \ge 1$ always.

Notice that since $a_i\leq s_i$ and $\lvert \tilde{a_i}+\alpha_is_i\rvert \leq n_i+1$, if we assume that $s_i\in [S_i,2S_i]$ say, there must exist some interval $I_i$, independent of $a_i$ and $s_i$, of length $O(n_i/S_i)$, which contains $a_i/s_i$. It follows from the discussion thus far that \eqref{eq-lem1} is bounded above by 
\[(\log y)^{O(1)}\sum_{\substack{n_1,\ldots,n_k\\ 1\leq n_i\leq y}}\prod_{i=1}^kn_i^{-1}\sup_{\substack{S_1,\ldots,S_k\\ 1\leq S_i\leq 2hn_i}} \Bigl(\sum_{\substack{s_1, \dots, s_k \\s_i\in [S_i,2S_i]}}\sum_{\substack{a_1, \dots, a_k \\ a_i/s_i \in I_i \\ \ogcd{a_i}{s_i} = 1 \\ \sum_i a_i/s_i \in \mathbb Z}}1\Bigr). \]
Noting that the term inside the brackets depends only on the size of each $n_i$, we can dyadically pigeonhole this sum according to the size of $n_i$ (at an additional cost of $(\log y)^{O(1)}$, and so it suffices to show that 

\[ \sup_{\substack{N_1,\ldots,N_k\\ 1\leq N_i\leq 2y}}\prod_{i=1}^kN_i^{-1}\sup_{\substack{S_1,\ldots,S_k\\ 1\leq S_i\leq 4hN_i}} \Bigl(\sum_{\substack{s_1, \dots, s_k \\s_i\in [S_i,2S_i]}}\sum_{\substack{a_1, \dots, a_k \\ a_i/s_i \in I_i \\ \ogcd{a_i}{s_i} = 1 \\ \sum_i a_i/s_i \in \mathbb Z}}1\Bigr)\ll (\log y)^{O(1)}h^{\lfloor k/2\rfloor},\]
where each $I_i$ is an interval of length $O(N_i/S_i)$ centred at $\alpha_i$. By Theorem~\ref{th-oddsum} the term inside the brackets is, for some $X\subseteq \{1,\ldots,k\}$ of size $\abs{X}=\lceil k/2\rceil$, 
\[\ll (\log y)^{O(1)}\prod_{i\in X}N_i\prod_{j\not\in X}S_j\ll (\log y)^{O(1)}h^{\lfloor k/2\rfloor} N_1\cdots N_k,\]
since each $S_i\ll hN_i$. This concludes the proof.
\end{proof}

In certain regimes the following more direct bound is more effective.
\begin{lemma}\label{lem:bounds-on-r-i-sums-no-Fs}
For a squarefree integer $q$ and a fixed integer $k \ge 1$,
\begin{equation*}
    \sum_{r_1, \dots, r_k |q} \frac{\mu(r_i)^2}{\phi(r_i)} \sum_{\substack{b_1, \dots, b_k \\ 1 \le b_i \le r_i \\ \ogcd{b_i}{r_i} = 1 \\ \sum_i b_i/r_i \in \mathbb Z }}1 \le \prod_{p|q} \left(1 + \frac{2^{k}}{p-1}\right).
\end{equation*}
\end{lemma}
\begin{proof}
The left-hand side is multiplicative in $q$ (since the constraints on the inside sum are independent for each prime $p|r_1 \cdots r_k$), and hence it suffices to prove this when $q$ is a prime, which we will call $p$. Then for all $i$, either $r_i = 1$ or $r_i = p$. Since $\sum_i \frac{b_i}{r_i} \in \mathbb Z$, either none or at least two of the $r_i$'s are equal to $p$. If no $r_i$ is equal to $p$, then $b_1 = \cdots = b_k = 1$, and this term contributes $1$ overall to the sum.

If $j$ of the $r_i$'s are equal to $p$, for $j \ge 2$, then the contribution of these terms to the sum is
\begin{equation*}
\le \binom{k}{j} \frac{\phi(p)^{j-1}}{\phi(p)^j} = \binom{k}{j} \frac 1{\phi(p)}.
\end{equation*}
To see this, note that there are $\binom{k}{j}$ possibilities for which of the $r_i$'s are equal to $p$. If we reorder such that $r_1 = \cdots = r_j = p$ and $r_{j+1} = \cdots = r_k = 1$, then $b_{j+1} = \cdots = b_k = 1$. There are $\le \phi(p)$ options for each of $b_1, \dots, b_{j-1}$, and then the value of $b_j$ is determined by the constraint that $\sum_i b_i/r_i \in \mathbb Z$. This is an inequality rather than an equality because if $\sum_{i=1}^{j-1} b_i/r_i \in \mathbb Z$, then $b_j$ cannot be relatively prime to $r_j$ and also satisfy the sum constraint.

The left-hand side (when $q=p$ is prime) is therefore at most
\begin{equation*}
1 + \frac 1{\phi(p)}\sum_{j=2}^k \binom{k}{j} \le 1 + \frac {2^k}{\phi(p)},
\end{equation*}
and the proof is complete.
\end{proof}

Although our primary interest is $M_k(q,h)$, for the latter stages of the proof for general $q$, it is convenient to bound a more general quantity for smooth $q$: for any integers $k_1,k_2\geq 0$ let
\[M_{k_1,k_2}(q,h) = \sum_{1\leq n\leq q}\brac{\sum_{\substack{n\leq m < n+h\\ \ogcd{m}{q}=1}}1-\frac{\phi(q)}{q}h}^{k_1}\brac{\sum_{\substack{n\leq m < n+h\\ \ogcd{m}{q}=1}}1}^{k_2}.\]
\begin{lemma}\label{thm:reduced-residues-squarefree-friable-modulus}
Let $h \ge 1$. Let $q \ge 1$ be a squarefree integer, and let $y \ge h$ be such that for any prime $p|q$, $p \le y$. Then for any integers $k_1,k_2\geq 0$,
\begin{equation*}
    M_{k_1,k_2}(q,h) \ll (\log y)^{O_{k_1,k_2}(1)}h^{\lfloor \frac{k_1}{2}\rfloor+k_2}q.
\end{equation*}
\end{lemma}
\begin{proof}
Generalising the proof of \cite[Lemma 2]{MV86} we can write $M_k(q,h) =  q(\frac{\phi(q)}{q})^k V_{k_1,k_2}(q,h)$, where (writing $k=k_1+k_2$ for brevity)
\begin{equation*}
V_{k_1,k_2}(q,h) := \subsum{q_1, \dots, q_{k_1} \\ 1 < q_i|q} \subsum{q_{k_1+1}, \dots, q_{k} \\ 1 \leq q_i|q} \left(\prod_{i=1}^k \frac{\mu(q_i)}{\phi(q_i)}\right)\subsum{a_1, \dots, a_k \\ 1 \le a_i \le q_i \\ \ogcd{a_i}{q_i} = 1 \\ \sum_i a_i/q_i \in \mathbb Z} E\left(\frac{a_1}{q_1}\right) \cdots E\left(\frac{a_k}{q_k}\right).
\end{equation*}
By the triangle inequality and the fact that $|E(x)| \le F(x)$ for all $x \in \mathbb R$, we immediately have that
\begin{equation*}
|V_{k_1,k_2}(q,h)| \le \subsum{q_1, \dots, q_{k_1} \\ 1 < q_i|q} \subsum{q_{k_1+1}, \dots, q_{k} \\ 1 \leq q_i|q} \prod_{i=1}^k \frac{\mu(q_i)^2}{\phi(q_i)}\subsum{a_1, \dots, a_k \\ 1 \le a_i \le q_i \\ \ogcd{a_i}{q_i} = 1 \\ \sum_i a_i/q_i \in \mathbb Z} F\left(\frac{a_1}{q_1}\right) \cdots F\left(\frac{a_k}{q_k}\right).
\end{equation*}

We claim that for each tuple $q_1, \dots, q_k$ in this sum there exist tuples $s_1, \dots, s_k$ and $r_1, \dots, r_k$ such that 
\begin{enumerate}
\item $q_i = s_ir_i$ for all $i$,
\item $s_i \leq y^{2k}$ for all $i$ and $1<s_i$ for $1\leq i\leq k_1$,
\item $\bracgcd{\prod_{i=1}^k s_i}{\prod_{i=1}^k r_i} = 1$, and 
\item whenever $q_i > h$, then $s_i > h$ as well, and 
whenever $q_i \le h$, we have $s_i = q_i$ and $r_i = 1$.
\end{enumerate}
To prove this, we construct $s_1, \dots, s_k$ and $r_1, \dots, r_k$ as follows. For each $i$ with $q_i \le h$, set $s_i = q_i$ and $r_i = 1$. For each $i$ with $q_i > h$, we claim that $q_i$ has a divisor $d_i$ with $h < d_i \leq y^2$. Since each prime factor of $q_i$ is less than $y$, $d_i$ can be constructed greedily: if $d \leq h$ and $d|q_i$, then there exists a prime factor $p|q_i$ with $p\nmid d$, in which case $pd \leq y^2$. Iterating this process yields a divisor $d_i|q_i$ with $h < d_i\leq y^2$. We then set 
\begin{equation*}
s_i = \bracgcd{\prod_{\substack{1 \le j \le k \\ q_j > h}} d_j \prod_{\substack{1 \le j \le k \\ q_j \le h}} q_j}{q_i}\quad\textrm{and}\quad r_i=\frac{q_i}{s_i},
\end{equation*}
so that by definition of the $d_i$ and since each $q_i$ is squarefree, $h < d_i \le s_i$,  and $s_i \le \prod_{\substack{1 \le j \le k \\ q_j > h}} d_j \prod_{\substack{1 \le j \le k \\ q_j \le h}} q_j \le y^{2k}$. The only remaining thing to check is that for all $i,j$ we have $\ogcd{s_i}{r_j}=1$. By definition of $r_j$ (and since $q_j$ is squarefree) it suffices to note that by definition if $p\mid s_i$ and $p\mid q_j$ then $p\mid s_j$.

Note that any fraction of the shape $\frac{a_i}{q_i}$ (with $\ogcd{a_i}{q_i}=1$) can be written uniquely as $\frac{b_i}{r_i}+\frac{c_i}{s_i}$ with $\ogcd{b_i}{r_i}=\ogcd{c_i}{s_i}=1$. We thus replace the sum over $q_1, \dots, q_k$ by a sum over all such $r_1, \dots, r_k$ and $s_1, \dots, s_k$. This can only increase the value of the sum. In particular,
\[|V_{k_1,k_2}(q,h)| \le \subsum{r_1, \dots, r_k|q} \prod_{i=1}^k \frac{\mu(r_i)^2}{\phi(r_i)}\subsum{b_1, \dots, b_k \\ 1 \le b_i \le r_i \\ \ogcd{b_i}{r_i} = 1\\ \sum_i b_i/r_i \in \mathbb Z} W(\tfrac{b_1}{r_1},\ldots,\tfrac{b_k}{r_k})\]
where
\[W(\tfrac{b_1}{r_1},\ldots,\tfrac{b_k}{r_k})=\subsum{s_1, \dots, s_k \\ s_i \leq y^{2k}\\ r_i>1\Rightarrow s_i>h} \brac{\prod_{1\leq i\leq k_1}1_{s_i>1}}\prod_{i=1}^k \frac{\mu(s_i)^2}{\phi(s_i)}\subsum{c_1, \dots, c_k \\ 1 \le c_i \le s_i \\ \ogcd{c_i}{s_i} = 1 \\ \sum_i c_i/s_i \in \mathbb Z} \prod_{i=1}^k F\left(\frac{c_i}{s_i} + \frac{b_i}{r_i}\right).\]
By Lemma \ref{lem:bounds-on-s-i-sums-with-alpha-shifts}, for any $k' \ge k_1$, the contribution to $W(\tfrac{b_1}{r_1},\ldots,\tfrac{b_k}{r_k})$ from those tuples with $k'$ of the $s_i$ being $>1$ is $\ll (\log y)^{O(1)}h^{\lfloor \frac{k'}{2}\rfloor+k-k'}$ (where we have used the trivial bound $F\leq h$ for those $s_i=1$). Since $k=k_1+k_2$ this is $\ll (\log y)^{O(1)}h^{\lfloor \frac{k_1}{2}\rfloor+k_2}$. Thus
\begin{align*}
|V_{k_1,k_2}(q,h)| &\ll (\log y)^{O(1)} h^{\lfloor \frac{k_1}{2}\rfloor+k_2}\subsum{r_1, \dots, r_k|q} \prod_{i=1}^k \frac{\mu(r_i)^2}{\phi(r_i)}\subsum{b_1, \dots, b_k \\ 1 \le b_i \le r_i \\ \ogcd{b_i}{r_i} = 1\\ \sum b_i/r_i \in \mathbb Z} 1 \\
&\ll (\log y)^{O(1)}h^{\lfloor \frac{k_1}{2}\rfloor+k_2}\prod_{p|q}\left(1 + \frac{2^k}{p-1}\right),
\end{align*}
by Lemma \ref{lem:bounds-on-r-i-sums-no-Fs}. Since $p|q$ only if $p \le y$, we have
\begin{equation*}
\prod_{p|q} \left(1 + \frac{2^k}{p-1}\right) \ll \left(\frac{q}{\phi(q)}\right)^{2^k} \ll (\log y)^{2^k},
\end{equation*}
which completes the proof.
\end{proof}

\subsubsection{The rough component} 

To complete the proof of Theorem~\ref{thm:montgomery-vaughan-odd} we use an argument of Montgomery and Vaughan \cite{MV86} to handle the `rough' component of $q$, which is divisible by only large primes $>y$. To this end, set $y = h^k$, and for any squarefree $q$, let $q=q_1q_2$ where $p\leq y$ for all primes $p\mid q_1$ and $p>y$ for all primes $p\mid q_2$. Let $P_1=\phi(q_1)/q_1$ and $P_2=\phi(q_2)/q_2$. As in \cite[Section 7]{MV86} we have
\begin{align*}
M_k(q,h) 
&= \sum_{1\leq n_1\leq q_1}\sum_{1\leq n_2\leq q_2}(P_2D_1(n_1)+D_2(n_1,n_2))^k\\
&=\sum_{k_1+k_2=k}\binom{k}{k_1}P_2^{k_1}\sum_{1\leq n_1\leq q_1}D_1(n_1)^{k_1}\sum_{1\leq n_2\leq q_2}D_2(n_1,n_2)^{k_2}
\end{align*}
where
\[D_1(n_1) = \sum_{1\leq m\leq h}1_{\ogcd{m+n_1}{q_1}=1}-hP_1\]
and
\[D_2(n_1,n_2)=\sum_{1\leq m\leq h}1_{\ogcd{m+n_1}{q_1}=1}1_{\ogcd{m+n_2}{q_2}=1}-P_2\sum_{1\leq m\leq h}1_{\ogcd{m+n_1}{q_1}=1}.\]
To conclude the proof of Theorem~\ref{thm:montgomery-vaughan-odd} therefore, it suffices to show that, for each $k_1,k_2\geq 0$ with $k_1+k_2=k$,
\begin{equation}\label{eq-dbound}
\sum_{1\leq n_1\leq q_1}D_1(n_1)^{k_1}\sum_{1\leq n_2\leq q_2}D_2(n_1,n_2)^{k_2}\ll (\log h)^{O(1)}qP_2^{-k_1}\brac{(hP_2)^{\frac{k-1}{2}}+hP_2}.
\end{equation}
(Here we have used $P_1\geq (\log h)^{-O(1)}$ to simplify the right-hand side.) We henceforth fix such $k_1$ and $k_2$. We require the following lemma to control the sum over $D_2$, which is a combination of \cite[Lemmas 9 and 10]{MV86}. 
\begin{lemma}\label{lem-mv2}
Let $q\geq 1$ and $P=\phi(q)/q$. Let $A$ be a set of at most $h$ integers, such that for each prime divisor $p\mid q$ the integers in $A$ are distinct modulo $p$, and $p\geq y$. For any $k\geq 1$,
\[\sum_{1\leq n\leq q}\brac{\sum_{m\in A}1_{\ogcd{m+n}{q}=1}-\abs{A}P}^k = q\sum_{1\leq s\leq \left\lfloor \frac k2\right\rfloor}F_{k,s}(P)\abs{A}^s+O\brac{q\frac{(hP)^k+hP}{y}},\]
where $F_{k,s}(P)$ are polynomials in $P$ of degree at most $k$ with $\abs{F_{k,s}(P)}\ll P^s$. 
\end{lemma}
Using Lemma~\ref{lem-mv2}, if 
\[A=A(n_1)=\{1\leq m\leq h : \ogcd{m+n_1}{q_1}=1\},\]
the left-hand side of \eqref{eq-dbound} is (using the trivial bound $\abs{D_1(n_1)}\ll h$ for the error term)
\[q_2\sum_{1\leq s\leq \lfloor k_2/2\rfloor}F_{k_2,s}(P_2)\sum_{n_1\leq q_1}D_1(n_1)^{k_1}\abs{A(n_1)}^s+O\brac{qh^{k_1}\frac{(hP_2)^{k_2}+hP_2}{y}}.\]
Since $y=h^k$, this error term is acceptable, as (recalling $k=k_1+k_2$)
\[h^{k_1}\frac{(hP_2)^{k_2}+hP_2}{y}=P_2^{k_2}+h^{1-k_2}P_2\ll hP_2.\]
Now note that, by Lemma~\ref{thm:reduced-residues-squarefree-friable-modulus}, 
\[\sum_{n_1\leq q_1}D_1(n_1)^{k_1}\abs{A(n_1)}^s=M_{k_1,s}(q_1,h)\ll (\log y)^{O(1)}h^{\lfloor \frac{k_1}{2}\rfloor+s}q_1.\]
It follows that the left-hand side of \eqref{eq-dbound} is 
\[\ll qhP_2+(\log h)^{O(1)}qh^{\lfloor \frac{k_1}{2}\rfloor}\sum_{1\leq s\leq \lfloor k_2/2\rfloor}F_{k_2,s}(P_2)h^s.\]
(Note that $P_1\leq (\log h)^{O(1)}$ and hence the first error term is dominated by the second summand of the right-hand side of \eqref{eq-dbound}.) Since $F_{k_2,s}(P)\ll P^s$ the sum here is $\ll (hP_2)^{\lfloor \frac{k_2}{2}\rfloor}+hP_2$. The required estimate \eqref{eq-dbound} now follows since
\[h^{\lfloor \frac{k_1}{2}\rfloor}\brac{(hP_2)^{\lfloor \frac{k_2}{2}\rfloor}+hP_2}\ll P_2^{-k_1}\brac{(hP_2)^{\frac{k-1}{2}}+hP_2}.\]
Indeed, if $h\leq P_2^{-1}$ then this is equivalent to
\[h^{\lfloor \frac{k_1}{2}\rfloor}\ll P_2^{-k_1},\]
which holds since $\lfloor k_1/2\rfloor \leq k_1$ for all $k_1\geq 0$, and if $h>P_2^{-1}$ this is equivalent to 

\[h^{\lfloor \frac{k_1}{2}\rfloor+\lfloor \frac{k_2}{2}\rfloor}P_2^{\lfloor \frac{k_2}{2}\rfloor}\ll h^{\frac{k-1}{2}}P_2^{\frac{k_2-k_1-1}{2}}.\]
This holds since $P_2\leq 1$ and $\lfloor k_2/2\rfloor \geq (k_2-1)/2$, and (since $k$ is an odd integer and $k_1+k_2=k$)
\[\lfloor\tfrac{k_1}{2}\rfloor+\lfloor \tfrac{k_2}{2}\rfloor=\frac{k-1}{2}.\]
This completes the proof of Theorem~\ref{thm:montgomery-vaughan-odd}.

\subsection{Proof of Theorem \ref{thm:conjecture-for-Rkh-for-k-odd}}

This proof will follow along the same basic lines as those Montgomery and Soundararajan \cite{MS04} used to estimate $R_k(h)$ when $k$ is even. Our goal is to estimate $R_k(h)$; recall that $R_k(h)$ is defined to be
\[R_k(h) := \sum_{\substack{1\leq d_1,\ldots,d_k\leq h\\ d_i\textrm{ distinct}}}\mathfrak{S}_0(d_1,\ldots,d_k).\]

Throughout this section, $\mathcal D = (d_1, \dots, d_k)$ will denote an ordered $k$-tuple of integers. Recall our convention that, if $\mathcal{J}=\{j_1<\cdots<j_t\}\subseteq \{1,\ldots,k\}$, then $\mathcal{D}_{\mathcal{J}}=(d_{j_1},\ldots,d_{j_t})$ is the $\abs{\mathcal{J}}$-tuple restricted to $\mathcal{J}$. If the elements $d_1, \dots, d_k$ are distinct, we will say that $\mathcal D$ is a \emph{distinct $k$-tuple}. Note that in fact $\mathfrak{S}$ and $\mathfrak{S}_0$ do not depend on the ordering of the tuple, although for technical reasons we find it convenient to work with ordered $k$-tuples throughout.

It is inconvenient that $\mathfrak S(\mathcal D)$ is convergent only if $\mathcal D$ is a distinct $k$-tuple. To overcome this obstacle, we define the singular series at a $k$-tuple $\mathcal D$ with respect to a modulus $q$, as defined in \cite{K21}, by
\begin{equation*}
\mathfrak S(\mathcal D;q) := \prod_{p|q}\left(1-\frac 1p\right)^{-k}\left(1-\frac{\nu_p(\mathcal D)}{p}\right),
\end{equation*}
where as usual $\nu_p(\mathcal D)$ is the number of distinct residue classes mod $p$ occupied by elements of $\mathcal D$. Note that this is a finite Euler product, so it is well-defined even if $\mathcal D$ has repeated elements. More precisely, this definition has the following immediate consequence for repeated elements:
\begin{lemma}\label{lem-S-relative-to-q-repeated-elements}
Let $\mathcal D_1$ be a $k_1$-tuple and let $\mathcal D_2$ be a $k_2$-tuple such that $1 \le k_1 < k_2$. Assume that $\mathcal D_1$ has distinct elements and that every element of $\mathcal D_1$ appears at least once in $\mathcal D_2$ (that is, $\mathcal D_1$ is a tuple whose elements are precisely the \emph{distinct} elements of $\mathcal D_2$ in some ordering). Then
\begin{equation*}
\mathfrak S(\mathcal D_2;q) = \left(\frac{q}{\phi(q)}\right)^{k_2-k_1} \mathfrak S(\mathcal D_1;q). 
\end{equation*}
\end{lemma}

Just as for $\mathfrak S(\mathcal D)$, we also define the refined singular series with respect to a modulus $q$, given by
\begin{equation*}
\mathfrak{S}_0(\mathcal{D}; q) = \sum_{\mathcal{Q}\subseteq \{1,\ldots,k\}}(-1)^{k-\abs{\mathcal{Q}}}\mathfrak{S}(\mathcal{D}_{\mathcal{Q}};q).
\end{equation*}
We will take the convention that $\mathfrak S(\emptyset) = \mathfrak S_0(\emptyset) = \mathfrak S(\emptyset;q) = \mathfrak S_0(\emptyset;q) = 1$ for all $q$. For appropriate choice of $q$, the sums $R_k(h)$ and $V_k(q,h)$ both have clean expansions in terms of the $\mathfrak S_0(\mathcal D;q)$, which we summarise in the following lemma.

\begin{lemma}
Let $h, k \ge 1$. Set $y = h^{k+1}$ and $q = \prod_{p \le y} p$. Then
\begin{equation}\label{eq:Vkqh-in-terms-of-S0q}
V_k(q,h) = \sum_{\substack{1 \le d_1, \ldots, d_k \le h}} \mathfrak S_0(\mathcal D;q),
\end{equation}
and
\begin{equation}\label{eq:Rkh-in-terms-of-S0q}
R_k(h) = \sum_{\substack{1 \le d_1, \ldots, d_k \le h \\ d_i \text{ distinct}}} \mathfrak S_0(\mathcal D;q) + O(1).
\end{equation}
\end{lemma}
\begin{proof}
The proof of this lemma is routine, depending on known expansions of $\mathfrak S_0$ and $V_k$. By \cite[equation (5) and Lemma 2.3]{K21},
\begin{equation}\label{eq:S0q-expand}
\mathfrak S_0(\mathcal D;q) = \sum_{\substack{q_1, \dots, q_k \\ q_i > 1 \\ q_i|q}} \prod_{i=1}^k \frac{\mu(q_i)}{\phi(q_i)} \sum_{\substack{a_1, \dots, a_k \\ 1 \le a_i \le q_i \\ \ogcd{a_i}{q_i} = 1 \\ \sum_i a_i/q_i \in \mathbb Z}} e\left(\sum_{i=1}^k \frac{d_ia_i}{q_i}\right).
\end{equation}
Equation \eqref{eq:S0q-expand}, along with the definition of $V_k(q,h)$ in \eqref{eq-defn-of-Vkqh}, immediately implies \eqref{eq:Vkqh-in-terms-of-S0q}. We now prove equation \eqref{eq:Rkh-in-terms-of-S0q}. By \cite[equations (50) and (51)]{MS04},
\begin{equation*}
R_k(h) = \sum_{\substack{1 \le d_1, \ldots, d_k \le h \\ d_i \text{ distinct}}} \sum_{\substack{q_1, \dots, q_k \\ q_i > 1 \\ q_i|q}} \prod_{i=1}^k \frac{\mu(q_i)}{\phi(q_i)} \sum_{\substack{a_1,\dots, a_k \\ 1 \le a_i \le q_i \\ \ogcd{a_i}{q_i} = 1 \\ \sum_i a_i/q_i \in \mathbb Z}}e\left(\sum_{i=1}^k \frac{d_ia_i}{q_i} \right) + O(1),
\end{equation*}
when $q = \prod_{p \le y} p$ and $y = h^{k+1}$, which along with \eqref{eq:S0q-expand} immediately implies \eqref{eq:Rkh-in-terms-of-S0q}. 
\end{proof}

Equations \eqref{eq:Vkqh-in-terms-of-S0q} and \eqref{eq:Rkh-in-terms-of-S0q} differ only (yet crucially) in that the $d_i$ in \eqref{eq:Rkh-in-terms-of-S0q} are required to be distinct, unlike in \eqref{eq:Vkqh-in-terms-of-S0q}. Our goal is to relate $R_k(h)$ to a sum over $V_j(q,h)$ for appropriate $j$ and $q$, and then apply Theorem \ref{thm:montgomery-vaughan-odd}. The argument to remove the distinctness condition is in essence a rather involved inclusion-exclusion argument; we will be able to express $R_k(h)$ as a linear combination of $V_j(q,h)$, for different values of $j$.

Set $\delta_{ij}(\mathcal{D}) = 1$ if $d_i = d_j$ and $\delta_{ij}(\mathcal{D}) = 0$ otherwise, so that
\begin{equation*}
1_{d_1,\ldots,d_k\textrm{ are distinct}}=\prod_{1 \le i < j \le k} (1-\delta_{ij}(\mathcal{D})) = \sum_{G\subseteq \binom{k}{2}}(-1)^{\abs{G}}\prod_{(i,j)\in G}\delta_{ij}(\mathcal{D}).
\end{equation*}
Note that the product $\prod_{(i,j)\in G}\delta_{ij}(\mathcal{D})$ in fact depends only on the connected components of the graph $G$, which is a partition of $\{1,\ldots,k\}$. If $\mathcal{P}$ is a partition of $\{1,\ldots,k\}$ then we write $G\sim \mathcal{P}$ to mean that $\mathcal{P}$ lists the connected components of $G$. Therefore, if
\begin{equation*}
\Delta_{\mathcal P}(\mathcal{D}) := \begin{cases} 1&\textrm{ if }d_i=d_j\textrm{ whenever }i,j\textrm{ are in the same cell of the partition }\mathcal{P}\textrm{ and }\\
0&\textrm{otherwise,}\end{cases}
\end{equation*}
then whenever $G\sim \mathcal{P}$ we have $\prod_{(i,j)\in G}\delta_{ij}(\mathcal{D})=\Delta_{\mathcal{P}}(\mathcal{D})$. We have therefore shown that, if 
\begin{equation*}
w(\mathcal P) := \sum_{G\sim \mathcal P} (-1)^{\abs{G}},
\end{equation*}
then 
\[1_{d_1,\ldots,d_k\textrm{ are distinct}}=\sum_{\mathcal{P}}w(\mathcal{P})\Delta_{\mathcal{P}}(\mathcal{D}).\]
In particular, inserting this into \eqref{eq:Rkh-in-terms-of-S0q} yields
\begin{align*}
R_k(h) &= \sum_{\mathcal P} w(\mathcal P) \sum_{\substack{d_1, \dots, d_k \\ 1 \le d_i \le h}}\Delta_{\mathcal{P}}(\mathcal D) \mathfrak S_0(\mathcal D;q)  + O(1).
\end{align*}
The tuples $\mathcal D$ are now (for all except the trivial partition into singletons) \emph{forced} to have repeated elements, but we can remove these repeated elements by applying Lemma \ref{lem-S-relative-to-q-repeated-elements}. However, Lemma \ref{lem-S-relative-to-q-repeated-elements} applies to $\mathfrak S(\mathcal D;q)$ and not $\mathfrak S_0(\mathcal D;q)$, so in preparation for applying Lemma \ref{lem-S-relative-to-q-repeated-elements}, we recall the definition of $\mathfrak{S}_0(\mathcal{D};q)$ to write
\begin{align*}
R_k(h) &=  (-1)^{k}\sum_{\mathcal P} w(\mathcal P) \sum_{\substack{d_1, \dots, d_k \\ 1 \le d_i \le h }}\Delta_{\mathcal{P}}(\mathcal{D})  \sum_{\mathcal Q \subseteq \{1,\ldots,k\}} (-1)^{|\mathcal Q|} \mathfrak S(\mathcal{D}_{\mathcal{Q}};q) + O(1).
\end{align*}
The task now is to rewrite this sum to decouple the dependence between the tuple $\mathcal{D}$ and the partition $\mathcal{P}$. 

For a fixed integer $\ell$, denote by $P_\ell(z)$ the polynomial given by
\begin{equation}\label{eq:Pkz-polynomial-definition}
P_\ell(z) := \frac{(1-z)^\ell-1}{z} = \sum_{j=1}^\ell \binom{\ell}{j} (-1)^{j} z^{j-1},
\end{equation}
so that for example $P_1(z) = -1$, $P_2(z) = z-2$, $P_3(z) = -z^2+3z-3$, and so on. The following lemma shows that $R_k(h)$ can be written in terms of the polynomials $P_\ell(z)$.

\begin{lemma}\label{lem:partitions-to-polynomial-coefficients-S}
Fix a partition $\mathcal P$ of $\{1, \dots, k\}$, and write $\mathcal P = \{S_1, \dots, S_M\}$. Then
\begin{align*}
\sum_{\substack{d_1, \dots, d_k \\ 1\le d_i \le h }}\Delta_{\mathcal{P}}(\mathcal{D}) &\sum_{\mathcal Q\subseteq \{1,\ldots,k\}} (-1)^{|\mathcal Q|} \mathfrak S(\mathcal{D}_{\mathcal Q};q) \\
&= \sum_{\substack{d_1, \dots, d_M \\ 1 \le d_i \le h}} \sum_{\substack{\mathcal J \subseteq \{1, \dots, M\}}}\prod_{m \in \mathcal J} P_{|S_m|}\left(\frac{q}{\phi(q)}\right) \mathfrak S(\mathcal{D}_{\mathcal J};q),
\end{align*}
where $P_\ell(z)$ is given in \eqref{eq:Pkz-polynomial-definition}.
\end{lemma}
\begin{proof}
By Lemma \ref{lem-S-relative-to-q-repeated-elements}, any multiple instance of any $d_i$ can be removed and replaced by a factor of $\frac{q}{\phi(q)}$, so the left-hand side can be expressed as a sum of the form
\begin{equation*}
\sum_{\substack{d_1, \dots, d_M \\ 1 \le d_i \le h}} \sum_{\substack{\mathcal J \subseteq \{1, \dots, M\}}} g_{\mathcal J}\left(\frac{q}{\phi(q)}\right) \mathfrak S(\mathcal{D}_{\mathcal J};q),
\end{equation*}
for some polynomials $g_{\mathcal J}$. Our aim is to show that for all $\mathcal J$, $g_{\mathcal J} = \prod_{m \in \mathcal J} P_{|S_m|}$. 

For a fixed $\mathcal J \subseteq \{1, \dots, M\}$, the terms contributing to $\mathfrak S(\mathcal{D}_{\mathcal J};q)$ are precisely those $\mathcal Q$ such that $|\mathcal Q \cap S_m| \ge 1$ if and only if $m \in \mathcal J$, and the contribution from $\mathcal Q$ is $\prod_{m \in \mathcal J} \left(\frac{q}{\phi(q)}\right)^{|\mathcal Q \cap S_m|-1} \mathfrak S(\mathcal{D}_{ \mathcal J};q)$. Thus, the coefficient of $\mathfrak S(\mathcal{D}_{\mathcal J};q)$ is given by
\begin{align*}
\sum_{\substack{\mathcal Q \subseteq \{1,\ldots,k\} \\ |\mathcal{Q}\cap S_m| \ge 1 \Leftrightarrow m \in \mathcal J }} (-1)^{|\mathcal Q|}&\prod_{m \in \mathcal J} \left(\frac{q}{\phi(q)}\right)^{|\mathcal Q \cap S_m|-1} \\
&= \prod_{m \in \mathcal J} \sum_{t_m = 1}^{|S_m|} \binom{|S_m|}{t_m}(-1)^{t_m} \left(\frac{q}{\phi(q)}\right)^{t_m-1}.
\end{align*}
This expression is precisely $\prod_{m \in \mathcal J} P_{|S_m|}\left(\frac{q}{\phi(q)}\right),$ as desired.
\end{proof}

Applying Lemma~\ref{lem:partitions-to-polynomial-coefficients-S} we thus have that $R_k(h)$ is equal to
\begin{align*}
&(-1)^{k}\sum_{\mathcal P = \{S_1, \dots, S_M\}} w(\mathcal P)\sum_{\substack{d_1, \dots, d_M \\ 1 \le d_i \le h}} \sum_{\substack{\mathcal J \subseteq \{1, \dots, M\}}} \prod_{m \in \mathcal J} P_{|S_m|}\left(\frac{q}{\phi(q)}\right) \mathfrak S(\mathcal{D}_{\mathcal J};q) + O(1).
\end{align*}
We now note that
\[\mathfrak{S}(\mathcal{D};q) = \sum_{\mathcal{R}\subseteq \{1,\ldots,k\}}\mathfrak{S}_0(\mathcal{D}_{\mathcal{R}};q).\]
Therefore 
\begin{align*}
\sum_{\substack{d_1, \dots, d_M \\ 1 \le d_i \le h}} &\sum_{\substack{\mathcal J \subseteq \{1, \dots, M\}}} \prod_{m \in \mathcal J} P_{|S_m|}\left(\frac{q}{\phi(q)}\right) \mathfrak S(\mathcal{D}_{\mathcal J};q) \\
&= \sum_{\substack{d_1, \dots, d_M \\ 1 \le d_i \le h}} \sum_{\substack{\mathcal R \subseteq \{1, \dots, M\}}} \brac{\sum_{\substack{\mathcal{J}\subseteq \{1,\ldots,M\}\\ \mathcal{J}\supseteq \mathcal{R}}}\prod_{m \in \mathcal J} P_{|S_m|}\left(\frac{q}{\phi(q)}\right)}\mathfrak{S}_0(\mathcal{D}_{\mathcal R};q) \\ 
&=\sum_{\substack{d_1, \dots, d_M \\ 1\le d_i \le h}} \sum_{\mathcal R \subseteq \{1, \dots, M\}} f_{\mathcal R,\mathcal P}\left(\frac{q}{\phi(q)}\right)\mathfrak S_0(\mathcal D_{\mathcal R};q),
\end{align*}
say, where
\begin{equation*}
f_{\mathcal R,\mathcal P}(z) = \prod_{m \in \mathcal R} P_{|S_m|}(z) \prod_{m \not\in \mathcal R} \left(1+P_{|S_m|}(z)\right)
\end{equation*}
is a polynomial of degree $k-M$, unless (recalling that $P_1(z)=-1$) there exists some $m\not\in \mathcal{R}$ such that $|S_m| = 1$, in which case $f_{\mathcal R, \mathcal P}(z) = 0$.

We are now ready to return to the original computation, whence
\begin{align*}
R_k(h) &= (-1)^{k}\sum_{\mathcal P} w(\mathcal P) \sum_{\substack{d_1, \dots, d_M \\ 1\le d_i \le h}} \sum_{\mathcal R \subseteq \{1, \dots, M\}} f_{\mathcal R,\mathcal P}\left(\frac{q}{\phi(q)}\right)\mathfrak S_0(\mathcal{D}_{\mathcal R};q) + O(1) \\
&= (-1)^{k}\sum_{\mathcal P} w(\mathcal P) \sum_{\substack{\mathcal R \subseteq \{1, \dots, M\}}} f_{\mathcal R,\mathcal P}\left(\frac{q}{\phi(q)}\right) h^{M-|\mathcal R|} V_{|\mathcal R|}(q,h)  + O(1).
\end{align*}
The second line follows by applying \eqref{eq:Vkqh-in-terms-of-S0q}; if $|\mathcal R| = 0$, the term $V_0(q,h)$ does not appear and in its place is simply $1$ (recalling that $\mathfrak S_0(\mathcal D_{\emptyset};q) = 1$), but for simplicity we will write $V_0(q,h) = 1$.

Recall that $q = \prod_{p \le y} p$, and that in the notation of Theorem \ref{thm:montgomery-vaughan-odd}, $M_j(q,h) = q\left(\frac{\phi(q)}{q}\right)^j V_j(q,h)$. By \cite{MV86} in the even case and Theorem \ref{thm:montgomery-vaughan-odd} in the odd case, $V_{\mathcal R}(q,h) \ll (\log h)^{O(1)} h^{\lfloor \abs{\mathcal{R}}/2\rfloor}$. Since $f_{\mathcal{R},\mathcal{P}} = 0$ unless $m\in \mathcal{R}$ whenever $\abs{S_m}=1$, and $f_{\mathcal{R},\mathcal{P}}(q/\phi(q))\leq (\log h)^{O(1)}$ always, we have shown that
\[R_k(h) \ll  (\log h)^{O(1)}\sum_{\mathcal P = \{S_1, \dots, S_M\}} \sum_{\substack{\mathcal R \subseteq \{1, \dots, M\}\\ \abs{S_m}=1\implies m\in \mathcal{R}}}  h^{M-\lceil|\mathcal R|/2\rceil}.\]
For a given partition $\mathcal P = \{S_1, \dots, S_M\}$, let $N_1$ denote the number of $m$ with $|S_m| = 1$, i.e.~ the number of parts of size $1$. The only $\mathcal R$ that contribute in this sum (for a fixed $\mathcal{P}$) have $|\mathcal R| \ge N_1$. For a fixed $N_1$, the total number of parts $M$ must satisfy
\begin{equation*}
M \le \frac{k-N_1}{2} + N_1 = \frac{k+N_1}{2}.
\end{equation*}
Therefore (since $\abs{\mathcal{R}}\geq N_1$ in the sum above)
\[R_k(h) \ll  (\log h)^{O(1)}\sum_{\mathcal P} h^{\lfloor\frac{k+N_1}{2}\rfloor-\lceil N_1/2\rceil}.\]
When $k$ is odd we have (for any $N_1$)
\[\left\lfloor\frac{k+N_1}{2}\right\rfloor-\lceil N_1/2\rceil= \frac{k-1}{2},\]
so
\[R_k(h) \ll (\log h)^{O(1)} h^{\frac{k-1}{2}}, \]
as desired.

\subsection{Formulae for $R_k(h)$}\label{subsec-rkformulae}

Our discussion in the proof of Theorem \ref{thm:conjecture-for-Rkh-for-k-odd} allows us to write a precise expression for the main term of $R_k(h)$. For $k$ even, an asymptotic for $R_k(h)$ has already been given in \cite{MS04}, dependent on $V_j(q,h)$ for $j \le k$ even. For odd $k$, however, every $V_j(q,h)$ for $j \le k$ contributes (we expect) to the main term.

As in the previous subsection, we take $q=\prod_{p \le y}p$, where $y=h^{k+1}$. The reader may be puzzled here, since we are estimating $R_k(h)$ (which has no parameter $q$) in terms of $V_j(q,h)$, and our proof in the previous section is valid if we take any $y>h^{k+1}$. We lose nothing by fixing some specific value of $y$, however, since we expect $V_k(q,h)$ to be independent of $q$ as $q \to \infty$ along primorials, as it becomes a closer and closer approximation of a (normalised) $k$th moment of primes in short intervals. In particular, the formulae for $R_k(h)$ we present here are, despite appearances, (conjecturally) independent of $q$.

As in the previous section we have
\begin{equation*}
R_k(h) = (-1)^{k}\sum_{\mathcal P} w(\mathcal P) \sum_{\substack{\mathcal R \subseteq \{1, \dots, M\}}} f_{\mathcal R,\mathcal P}\left(\frac{q}{\phi(q)}\right) h^{M-|\mathcal R|} V_{|\mathcal R|}(q,h)  + O(1),
\end{equation*}
where we recall that $f_{\mathcal R, \mathcal P} = 0$
unless $m \in \mathcal R$ for every $m$ with $|S_m| = 1$. As above, for a given partition $\mathcal P = \{S_1, \dots, S_M\}$, let $N_1$ denote the number of $m$ with $|S_m| = 1$, i.e.~ the number of parts of size $1$. Again, we must have $|\mathcal R| \ge N_1$, and $M \le \frac{k+N_1}{2}$.

We expect $V_{\abs{\mathcal{R}}}(q,h)\asymp h^{\lfloor \abs{\mathcal{R}}/2\rfloor+o(1)}$, and therefore for a fixed pair $\mathcal{P},\mathcal{R}$ the corresponding summand has magnitude $\asymp  h^{M-\lceil\lvert\mathcal R\rvert/2\rceil+o(1)}$. Since $\lvert \mathcal{R}\rvert \geq N_1$, for a fixed partition $\mathcal P$ the dominant term in the inner sum will be from those $\mathcal{R}$ with $\lvert \mathcal{R}\rvert=N_1$ if $N_1$ is even, or $\lvert \mathcal{R}\rvert\{N_1,N_1+1\}$ if $N_1$ is odd. The corresponding contribution therefore from the summand corresponding to $\mathcal{P}$ has size
\begin{equation*}
\ll (\log h)^{O(1)}h^{M-\lceil N_1/2\rceil}.
\end{equation*}
If there are, in general, $N_i$ parts of the partition with size $i$, then $k=N_1+2N_2+3N_3+\cdots$ and 
\[M=N_1+N_2+\cdots=\frac{k}{2}+\frac{N_1}{2}-\frac{1}{2}N_3-N_4-\cdots\]
The summand corresponding to $\mathcal{P}$ therefore has size
\begin{equation}\label{eq-pterm}
\ll (\log h)^{O(1)}h^{\frac{k}{2}-\{ N_1/2\}-\frac{1}{2}N_3-N_4-\cdots}.
\end{equation}
If $k=2\ell$ is even, this is maximised when $N_1$ is even and $N_3=N_4=\cdots=0$. Therefore, the only summands which contribute to the main term for $R_k(h)$ are those corresponding to partitions into parts of sizes $1$ and $2$, each of which contributes a summand of magnitude $\asymp h^{k/2+o(1)}$. This argument (which is essentially that used by Montgomery and Soundararajan \cite{MS04}) establishes that
\begin{align*}
R_{2\ell}(h) &\sim \sum_{\substack{\mathcal P = \{S_1, \dots, S_M\} \\ |S_m| = 1 \text{ or } 2 \\ \mathcal R = \{m:S_m = 1\}}} w(\mathcal P) f_{\mathcal R,\mathcal P}\left(\frac{q}{\phi(q)}\right) h^{M-|\mathcal R|} V_{|\mathcal R|}(q,h) \\
&\sim \sum_{j=0}^{\ell} (-1)^j \binom{2\ell}{2j}\frac{(2j)!}{j!2^j} \left(\frac{q}{\phi(q)}-1\right)^{j} h^{j}V_{2\ell-2j}(q,h),
\end{align*}
where in the second line $j$ counts the number of parts of size $2$ in $\mathcal P$, so that $M-|\mathcal R| = j$. Note that there are $\binom{2\ell}{2j}\frac{(2j)!}{j!2^j}$ partitions of $[1,2\ell]$ with $j$ parts of size $2$ and the rest of size $1$, and for these partitions $w(\mathcal P) = (-1)^j$. Here we have used that
\begin{align*}
f_{\mathcal R,\mathcal P}(z) 
&= \prod_{m \in \mathcal R} P_{|S_m|}(z) \prod_{m \not\in \mathcal R} \left(1+P_{|S_m|}(z)\right)\\
&=P_1(z)^{M-j}(1+P_2(z))^{j}=(-1)^k(z-1)^{j}
\end{align*}
when $\mathcal{P}$ is a partition of $\{1,\ldots,k\}$ into $M-j$ parts of size $1$ and $j$ parts of size $2$, and $\mathcal{R}=\{m: \abs{S_m}=1\}$. 

If $k=2\ell+1$ then it is slightly more complicated. Note that (since $\mathcal{P}$ is a partition of $\{1,\ldots,k\}$) we cannot have $N_1$ even and $N_3=N_4=\cdots=0$, and hence the greatest that \eqref{eq-pterm} can be is $\asymp h^{\frac{k-1}{2}+o(1)}$. This is achieved when either $N_1$ is odd and $N_3=N_4=\cdots=0$, or when $N_1$ is even and $N_3=1$ and $N_4=\cdots=0$.

In other words, the contributions to the main term for $R_k(h)$ (all of which with magnitude $\asymp h^{\frac{k-1}{2}+o(1)}$) arise from those partitions into parts of size only $1$ and $2$ and at most one part of size $3$. If there is a single part of size $3$ then the only significant inner summand is where $\mathcal{R}=\{m : \abs{S_m}=1\}$, but if there are no parts of size $3$ then any $\mathcal{R}$ which contains $\{m : \abs{S_m}=1\}$ of size either $N_1$ or $N_1+1$ must be counted.

To summarise, this argument shows that 
\begin{align*}
R_{2\ell+1}(h) &= -\sum_{\substack{\mathcal P = \{S_1, \dots, S_M\} \\ |S_m| = 1 \text{ or } 2 \\ \mathcal R = \{m:S_m = 1\}}} w(\mathcal P)  f_{\mathcal R,\mathcal P}\left(\frac{q}{\phi(q)}\right)h^{M-|\mathcal R|}V_{|\mathcal R|}(q,h) \\
&- \sum_{\substack{\mathcal P = \{S_1, \dots, S_M\} \\ |S_m| = 1 \text{ or } 2 \\ \mathcal R = \{m:S_m = 1\}}} w(\mathcal P)  \sum_{m \not\in \mathcal R} f_{\mathcal R \cup \{m\},\mathcal P}\left(\frac{q}{\phi(q)}\right) h^{M-|\mathcal R|-1} V_{|\mathcal R|+1}(q,h) \\
&-\sum_{\substack{\mathcal P = \{S_1, \dots, S_M\} \\ |S_m| = 1, 2 \text{ or } 3 \\ \text{one $m$ s.t. $|S_m| = 3$} \\ \mathcal R = \{m:|S_m| = 1\} \\ \mathcal R \ne \emptyset}} w(\mathcal P) f_{\mathcal R,\mathcal P} \left(\frac{q}{\phi(q)}\right) h^{M-|\mathcal R|} V_{|\mathcal R|}(q,h) + O((\log h)^{O(1)}h^{\ell-1}).
\end{align*}
This is equal to
\begin{align*}
&= \sum_{j=0}^{\ell} (-1)^{j} \binom{2\ell+1}{2j}\frac{(2j)!}{j!2^j}\left(\frac{q}{\phi(q)}-1\right)^jh^jV_{2\ell+1-2j}(q,h) \\
&+ \sum_{j=1}^{\ell} (-1)^{j} \binom{2\ell+1}{2j}\frac{(2j)!}{j!2^j}j\left(\frac{q}{\phi(q)}-1\right)^{j-1} \left(\frac{q}{\phi(q)} - 2\right) h^{j-1}V_{2\ell+2-2j}(q,h) \\
&+\sum_{j=0}^{\ell-1} 2(-1)^j\binom{2\ell+1}{3,2j}\frac{(2j)!}{j!2^j} \left(\frac{q}{\phi(q)}-1\right)^{j+1}\left(\frac{q}{\phi(q)}-2\right)h^{j+1}V_{2\ell-2-2j}(q,h) \\
&+ O((\log h)^{O(1)}h^{\ell-1}),
\end{align*}
calculating that if $\mathcal{P}$ is a partition of $\{1,\ldots,k\}$ into $j$ parts of size $2$, $1$ part of size $3$, and $k-2j-3$ parts of size $1$, and if $\mathcal{R}=\{m: \abs{S_m}=1\}$, then
\[f_{\mathcal R,\mathcal P}(z) = (-1)^{k-2j-3}(1+P_2(z))^{j}(1+P_3(z))=-(z-1)^{j+1}(z-2).\]
Furthermore, for such a partition we have $w(\mathcal{P})=2(-1)^{j}$. 

For example, when $\ell = 1$, we have proved that, recalling that $V_1=0$ and $V_0=1$, 
\begin{align*}
R_3(h) &= V_3(q,h) - 3\left(\frac{q}{\phi(q)}-2\right) V_2(q,h) \\
&+ 2\brac{\frac{q}{\phi(q)}-1}\brac{\frac{q}{\phi(q)}-2}h+O((\log h)^{O(1)}).
\end{align*}
We expect all three main terms here to have magnitude $\asymp (\log h)^{2}h$, since $q/\phi(q)\asymp \log h$. Similarly, when $\ell = 2$, we have
\begin{align*}
R_5(h) &= V_5(q,h) -10 \left(\frac{q}{\phi(q)}-2\right)V_4(q,h) - 10 \left(\frac{q}{\phi(q)}-1\right)hV_3(q,h) \\
&+ 35 \left(\frac{q}{\phi(q)}-1\right)\left(\frac{q}{\phi(q)}-2\right)hV_2(q,h) 
-20\brac{\frac{q}{\phi(q)}-1}^2\brac{\frac{q}{\phi(q)}-2}h^2
\\
&+O((\log h)^{O(1)}h).
\end{align*}
We expect each of these main terms to be of size $(\log h)^3h^2$. 

Although our proofs of these `asymptotic formulae' are unconditional, they are of limited use without an asymptotic understanding of the $V_j(q,h)$ terms. The main point is really that an asymptotic understanding of $R_k(h)$ and an asymptotic understanding of $V_k(q,h)$ are essentially equivalent (and both are roughly equivalent in turn to an asymptotic understanding of the $k$th moment of primes in short intervals).

\section{Further remarks}\label{sec-problems}

In this section, we will discuss further directions and open problems, focusing mainly on our bound on fractions. For simplicity we will discuss Theorem \ref{th-simpfrac}, although the same discussion could be applied to the more complicated Theorem \ref{th-oddsum}.

\subsection{Improving Theorem~\ref{th-simpfrac}}
The bound that we achieve in Theorem \ref{th-simpfrac} is best possible up to the factors of $\log Q$: elementary number theory shows that, for any $1\leq n\leq Q$, the number of $1\leq q\leq Q$ and $0\leq a_1,a_2\leq n$ such that 
\[\ogcd{a_1}{q}=\ogcd{a_2}{q}=1\]
is $\gg n^2Q$. Using the trivial identity
\[\frac{a_1}{q}+\frac{a_2}{q}-\frac{a_1+a_2}{q}=0\]
therefore produces $\gg n^2Q$ many solutions when $k=3$. (Of course it may not be true that $\ogcd{a_1+a_2}{q}=1$, but clearly $\tfrac{a_1+a_2}{q}=\tfrac{a'}{q'}$ for some $0\leq a'\leq n$ and $1\leq q'\leq Q$ with $\ogcd{a'}{q'}=1$.) A similar lower bound of $\gg n^{\frac{k+1}{2}}Q^{\frac{k-1}{2}}$ for other odd $k>3$ follows by adjoining such triples with diagonal pairs $(\tfrac{a}{q},-\tfrac{a}{q})$.

By a more delicate argument an improved lower bound can be obtained. For example, Blomer and Br\"{u}dern \cite{BB} prove that when $k=3$ there are $\gg (\log Q)^{3}n^2Q$ many solutions (and this exponent of $3$ is optimal given their coprimality assumptions). 

An asymptotic formula for the number of such solutions is not known for any odd $k$, even for $k=3$, although Blomer, Br\"{u}dern, and Salberger \cite{BBS} proved such a formula in a similar situation (differing only in the coprimality condition) when $n=Q$ and $k=3$, and Destagnol \cite{De} has proved such a formula for arbitrary $k$, again when $n=Q$ and under weaker coprimality conditions. It would be of interest to obtain even a reasonable conjecture for what the main term should be; the problem in this generality seems to be out of reach of Manin's conjecture.

Just as with the case of even $k$, we expect the count for odd $k$ to be dominated by `degenerate' solutions (for example, when $k=5$ the dominant contribution is likely from those where some subset of three fractions vanishes). It is perhaps then most natural to count only the `non-degenerate' solutions. 

More precisely, we say that a solution to
\[\frac{a_1}{q_1} + \cdots + \frac{a_k}{q_k} = 0\]
with $\ogcd{a_i}{q_i} = 1$, $0 < |a_i| \leq n$, and $1 \leq q_i \leq Q$ is \emph{degenerate} if there exists some proper subset $\emptyset \subsetneq I \subsetneq \{1, \ldots, k\}$ such that
\[\sum_{i \in I} \frac{a_i}{q_i} = 0.\]
For example, the lower bound described above for odd $k > 3$ is degenerate. Heuristically, one might expect the number of non-degenerate solutions to be of order $\approx n^{k-1}Q$ (without worrying about powers of $\log$). To see this,\footnote{We thank Tim Browning for showing us this heuristic.} we begin with $n^kQ^k$ total choices of $a_i$ and $q_i$, and we require that 
\[\sum_{i=1}^k a_i \prod_{j \ne i} q_j = 0 \]
(after clearing denominators). The left-hand side is an integer of size $O(nQ^{k-1})$, so the probability that it is $0$ is roughly $\frac 1{nQ^{k-1}}$, leading to $\frac{n^kQ^k}{nQ^{k-1}} = n^{k-1}Q$ solutions. This heuristic is of the same flavor as some of the reasoning behind Manin's conjecture, and Heath--Brown's result \cite{H-B} on counting the error term in solutions when $k = 4$ and $n = 1$ is consistent with this heuristic.

We are thus led to the following conjecture.
\begin{conjecture}\label{conj1}
Let $1\leq n\leq Q$ and $k\geq 3$ be an odd integer. The number of non-degenerate solutions to
\[\frac{a_1}{q_1}+\cdots+\frac{a_{k}}{q_{k}}= 0\quad\textrm{ where }\quad\ogcd{a_i}{q_i}=1\textrm{ for }1\leq i\leq k,\] 
with $\abs{a_i}\leq n$ and $q_i\leq Q$ for $1\leq i\leq k$, is
\[\ll (\log Q)^{O(1)} n^{k-1}Q.\]
\end{conjecture}

A proof of such a result would help significantly in our asymptotic understanding of $M_k(q,h)$ and $R_k(h)$.
\subsection{Improved error bounds for even moments}
There is a strong connection between the study of odd moments and the study of lower-order terms in even moments; indeed, in the original work of Montgomery and Vaughan \cite{MV86} both were treated essentially identically, resulting in the same bound (namely an exponent of $\frac{k}{2}-\frac{1}{7k}$) for both. 

Surprisingly, the methods of this paper do not appear to help with the lower order terms for even moments. A variant strong enough to address Conjecture~\ref{conj1} would be helpful towards this end. Improved bounds for the error terms in even moments would have a number of applications, such as to our understanding of primes in short intervals (following Montgomery and Soundararajan \cite{MS04}) and the distribution of squarefree numbers in short intervals (following Gorodetsky, Mangerel, and Rodgers \cite{GMR}). 

\appendix
\section{Relative greatest common divisors}\label{app1}
In this appendix we give a precise definition for, and prove the necessary properties of, the notion of relative greatest common divisors. Recall that this is the parameterisation of an arbitrary $k$-tuple of natural numbers $q_1,\ldots,q_k$ by some $2^k$ integers $g_I$, as $I$ ranges over subsets of $\{1,\ldots,k\}$, such that $q_i=\prod_{I\ni i}g_I$ and $\ogcd{g_I}{g_J}=1$ unless $I\subseteq J$ or $J\subseteq I$. We note that we always have $g_\emptyset=1$ for any $k$-tuple, but it is often convenient for technical reasons to allow $I$ to range over all subsets.

This is a generalisation of the notion of greatest common divisors, which can be formulated as the parameterisation of arbitrary pairs of natural numbers $q_1,q_2$ by three integers $g_1,g_2,g_{12}$ where $q_1=g_1g_{12}$ and $q_2=g_2g_{12}$ and $\ogcd{g_1}{g_2}=1$ -- indeed, we simply choose $g_{12}=\ogcd{q_1}{q_2}$. 

The appropriate generalisation to arbitrary $k$-tuples was first given by Dedekind \cite{Dede} in 1897. This concept has been rediscovered several times; we refer to the discussion in the thesis of Elsholtz \cite{El} for further historical remarks (and applications to the study of Egyptian fractions).

It is most common to define relative greatest common divisors recursively, beginning with $I=\{1,\ldots,k\}$, where
\[g_{\{1,\ldots,k\}} = \mathrm{gcd}(q_1,\ldots,q_k),\]
and in general, for $\emptyset \neq I \subsetneq \{1,\ldots,k\}$, define
\[g_I = \frac{\gcd(q_i : i\in I)}{\prod_{J\supsetneq I}g_J},\]
so that
\[\prod_{J \supseteq I} g_I = \gcd(q_i:i \in I).\]
The following equivalent `local' definition is more explicit and can be easier to manipulate in practice.
\begin{definition}
For any integers $q_1,\ldots,q_k$ and non-empty $I\subseteq \{1,\ldots,k\}$ we define the \emph{relative greatest common divisor} $g_I$ as follows. It suffices to define $\nu_p(I)$ for all primes $p$ and all $\emptyset\neq I\subseteq \{1,\ldots,k\}$, where $\nu_p(I)$ denotes the greatest exponent such that $p^{\nu_p(I)}$ divides $g_I$. For this, without loss of generality, we may assume that the $q_i$ are ordered such that 
\[\nu_p(q_1)\leq \cdots \leq \nu_p(q_k).\]
With this ordering, we define $\nu_p(\{1,\ldots,k\})=\nu_p(q_1)$ and for all $2\leq i\leq k$ set
\[\nu_p(\{i,\ldots,k\})=\nu_p(q_i)-\nu_p(q_{i-1}),\]
and set $\nu_p(I) = 0$ for all other sets $I$. (Note that $\nu_p(I)\geq 0$ for all $p$ and $I$, and so the $g_I$ are indeed integers.)
\end{definition}
For example, if $q_1=6$ and $q_2=9$ and $q_3=12$ then $2$-adically we have $\nu_2(\{1,2,3\})=0$ and $\nu_2(\{1,3\})=\nu_2(\{3\})=1$ (and otherwise $\nu_2(I)=0$), and $3$-adically we have $\nu_3(\{1,2,3\})=\nu_3(\{2\})=1$ (and otherwise $\nu_3(I)=0$), and therefore
\[g_1=g_{12}=g_{23}=1\quad g_3=g_{13}=2\quad g_2=g_{123}=3.\]
The following lemma captures all the basic properties of relative greatest common divisors that we will use (in fact we will never use the third, but it is worth recording nonetheless). Note that the first property is equivalent to the relation for the recursive definition of relative gcds, when applied to $I = \{i\}$.
\begin{lemma}
Let $q_1,\ldots,q_k\geq 1$ be any integers, and define the corresponding $g_I$ for $I\subseteq \{1,\ldots,k\}$ as above. We have:
\begin{enumerate}
\item for any $1\leq i\leq k$
\[q_i=\prod_{i\in I}g_I,\]
\item if $\ogcd{g_I}{g_J}>1$ then either $I\subseteq J$ or $J\subseteq I$, and 
\item if all $q_i$ are squarefree then $\ogcd{g_I}{g_J}=1$ for all $I\neq J$.
\end{enumerate}
\end{lemma} 
\begin{proof}
It suffices to give local proofs of each property, so fix some prime $p$ and order the $q_i$ such that 
\[\nu_p(q_1)\leq \cdots \leq \nu_p(q_k).\]
The first property follows from the fact that, for $1\leq i\leq k$,
\begin{align*}
\nu_p(q_i) 
&= (\nu_p(q_i)-\nu_p(q_{i-1}))+\cdots+(\nu_p(q_2)-\nu_p(q_1))+\nu_p(q_1)\\
&= \sum_{1\leq j\leq i} \nu_p(\{ j,j+1,\ldots,k\})\\
&=\sum_{i\in I}\nu_p(I),
\end{align*}
since $\nu_p(I)=0$ if $I$ is not of the shape $\{j,j+1,\ldots,k\}$ for some $1\leq j\leq k$. 

For the second property, note that if $\min(\nu_p(I),\nu_p(J))>0$ then there must exist some $i,j$ such that $I=\{i,\ldots,k\}$ and $J=\{j,\ldots,k\}$, whence $J\subseteq I$ or $I\subseteq J$ as required. 

Finally, the third property follows from the observation that if $\nu_p(q_i)\in \{0,1\}$ for all $1\leq i\leq k$ then by definition there is exactly one $I$ such that $\nu_p(I)\neq 0$ (namely if $0=\nu_p(q_1)=\cdots \nu_p(q_i)$ and $1=\nu_p(q_{i+1})=\cdots =\nu_p(q_k)$ then $I=\{i+1,\ldots,k\}$). 
\end{proof}

\end{document}